\documentclass[12pt]{amsart}
\usepackage{amssymb,amsmath,amsfonts,epsfig,latexsym, xypic}
\usepackage{enumerate}
\usepackage{tikz}

\tikzset{my_dot/.style={fill, circle, inner sep=0pt,minimum size=3pt}}
\tikzset{my_node/.style={fill, circle, inner sep=0pt,minimum size=3pt}}
\tikzset{inv/.style={fill, circle, inner sep=0pt,minimum size=0pt}}

\newtheorem{theorem}{Theorem}[section]
\newtheorem{definition}[theorem]{Definition}
\newtheorem{proposition}[theorem]{Proposition}
\newtheorem{lemma}[theorem]{Lemma}
\newtheorem{corollary}[theorem]{Corollary}

\newtheorem{claim}[theorem]{Claim}

\theoremstyle{definition}
\newtheorem{remark}[theorem]{Remark}
\newtheorem{example}[theorem]{Example}

\def\A{\ensuremath{\mathbb{A}}}
\def\C{\ensuremath{\mathbb{C}}}

\def\G{\ensuremath{\mathbb{G}}}

\def\Q{\ensuremath{\mathbb{Q}}}
\def\R{\ensuremath{\mathbb{R}}}
\def\Z{\ensuremath{\mathbb{Z}}}

\def\cA{\ensuremath{\mathcal{A}}}
\def\cD{\ensuremath{\mathcal{D}}}
\def\cM{\ensuremath{\mathcal{M}}}

\def\cX{\ensuremath{\mathcal{X}}}
\def\cY{\ensuremath{\mathcal{Y}}}

\DeclareMathOperator{\Aut}{Aut}

\DeclareMathOperator{\colim}{colim}
\DeclareMathOperator{\Gr}{Gr}
\DeclareMathOperator{\Hom}{Hom}
\DeclareMathOperator{\Mod}{Mod}
\DeclareMathOperator{\Out}{Out}

\DeclareMathOperator{\res}{res}
\DeclareMathOperator{\SL}{SL}

\DeclareMathOperator{\val}{val}

\def\col{\colon}
\def\Dgn{\Delta_{g,n}}
\def\Dn{\Delta_{1,n}}
\def\Gmw{{\bf G}}
\def\inj{\mathrm{inj}}
\def\Mgn{M^{\trop}_{g,n}} 
\def\ocM{\overline{\cM}}
\def\ocX{\overline{\cX}}
\def\ov{\overline}
\def\rep{\Delta^{\mathrm{rep}}_{g,n}}
\def\trop{\mathrm{trop}}

\newcommand{\double}{\genfrac..{0pt}1
{\raise -2pt\hbox{$\scriptstyle\longrightarrow$}}{\raise 4pt\hbox
{$\scriptstyle\longrightarrow$}}}

\newcommand{\FFF}[1]{\widetilde{D}([#1])}
\newcommand{\FFFF}[2]{\widetilde{D}_{#1}([#2])}

\newcommand{\EE}{\widetilde{D}}
\newcommand{\EEE}[1]{\widetilde{D}^{[#1]}}
\newcommand{\EEEE}[2]{\widetilde{D}_{#1}^{[#2]}}

\begin{document}

\title[The tropicalization of the moduli space of curves II]{The tropicalization of the moduli space of curves II: Topology and applications}

\author{Melody Chan}
\email{mtchan@math.brown.edu}

\author{S{\o}ren Galatius}
\email{galatius@stanford.edu}

\author{Sam Payne}
\email{sam.payne@yale.edu}

\begin{abstract}
We study the topology of the tropical moduli space parametrizing stable tropical curves of genus $g$ with $n$ marked points in which the bounded edges have total length $1$, and prove a bound on its connectivity. Using the identification of this space with the dual complex of the boundary in the moduli space of stable algebraic curves, we give a simple expression for the top weight cohomology of $\cM_{1,n}$ as a representation of the symmetric group and describe an explicit dual basis in homology consisting of abelian cycles for the pure mapping class group.
 \end{abstract}

\maketitle

\tableofcontents

\section{Introduction}

In this paper, we study the topology of a space $\Dgn$ that parametrizes isomorphism classes of certain vertex-weighted metric graphs with marked points.  The space $\Dgn$ has several natural geometric interpretations: 

\begin{enumerate}
\item as the subspace of the tropical moduli space $\Mgn$ parametrizing tropical curves of genus $g$ with $n$ marked points and total bounded edge length 1;

\item in the absence of marked points, as the topological quotient of the simplicial completion of Culler--Vogtmann outer space, as described in \cite[\S 2.2]{Vogtmann15}, by the action of the outer automorphism group $\Out(F_g)$; 

\item as the topological quotient of Harer's complex of curves on an $n$-punctured surface of genus $g$ by the action of the pure mapping class group; and

\item as the dual complex of the boundary divisor in the algebraic moduli space $\ocM_{g,n}$.  
\end{enumerate}

Combining these interpretations is particularly fruitful.  For instance, using interpretations (1) and (4), we describe and study interesting subcomplexes of the boundary complex of $\ocM_{g,n}$ as parameter spaces for tropical curves with specified properties, such as the property of having {\em repeated markings} (see \S\ref{sec:contractibility}).
This identification also endows the topological invariants of $\Dgn$ with interpretations as algebraic invariants of $\cM_{g,n}$.  For instance, important to our applications is the fact that the reduced rational homology of $\Dgn$ is canonically identified with the top graded piece of the weight filtration on the rational cohomology of the complex moduli space $\cM_{g,n}$.  (This is one instance of a general identification that we recall in \S\ref{sec:topweight} and the Appendix).  This identification facilitates computations such as Theorem~\ref{thm:cohomology}, which we prove using the contractibility of the repeated marking locus in $\Dgn$ (Theorem~\ref{thm:contractible}).  

As another example, using the weight spectral sequence on the logarithmic de Rham complex for $\cM_{g,n}$ and its degeneration at $E_2$, we describe all relations among abelian cycles for the pure mapping class group in middle homological degree in terms of coboundaries on $\Dgn$.  See Theorem~\ref{thm:abelianrelations}.

We now discuss five of our main results, starting with three about the topology of $\Dgn$.  
 
\begin{theorem}\label{thm:contractible}
For all $g>0$ and $n>1$, the subcomplex of $\Dgn$ parametrizing tropical curves with repeated markings is contractible.
\end{theorem}

\noindent  As one application of the contractibility of this modular subcomplex, we determine the homotopy type of $\Delta_{1,n}$, as follows.
 
\begin{theorem} \label{thm:boundary}  Both $\Delta_{1,1}$ and $\Delta_{1,2}$ are contractible.  For $n \geq 3$, the tropical moduli space $\Dn$ is homotopy equivalent to a wedge sum of $(n-1)! / 2$ spheres of dimension $n-1$.
\end{theorem}

\noindent An analogous result for $g = 0$ is well-known.  The moduli space $\Delta_{0,n}$ parametrizing trees with $n$ marked leaves is homotopic to a wedge sum of $(n-2)!$ spheres of dimension $n-4$ \cite{Vogtmann90}.  Note, however, that $\Dgn$ is not homotopic to a wedge sum of top-dimensional spheres in general.  The complex $\Delta_{2,n}$ can have rational homology in both of the top two degrees \cite[Theorem 1.4]{Chan15}, and can also have torsion in its integral homology \cite[forthcoming]{Chan15}.

For $g > 1$, the contractibility of the repeated marking locus is not enough to determine the homotopy type of the boundary complex.  Nevertheless, using Theorem~\ref{thm:contractible}, we can deduce that the boundary complex is highly connected for large $n$.  Recall that a topological space is $k$-connected if it is path-connected and its homotopy groups $\pi_i$ vanish for $1 \leq i \leq k$.

\begin{theorem}  \label{thm:connected}
For $g>1$, the tropical moduli space $\Dgn$ is $(n - 5g + 4)$-connected.
\end{theorem}

\noindent  The bound given by Theorem~\ref{thm:connected} is not sharp in general.  For instance, the first author has shown that $\Delta_{2,n}$ is $n$-connected \cite[Theorem~1.2]{Chan15}.  %Vanishing results for rational homology are consistent with the possibility that $\Delta_{g,n}$ may also be more highly connected for $g > 2$.  

%In general, Harer's bound on the virtual cohomological dimension of $\cM_{g,n}$ \cite{Harer86} implies that $\Dgn$ has reduced rational homology supported in the top $g - \delta_{0,n}$ degrees for $g > 0$, where $\delta_{ij}$ is the Kronecker delta function (see Proposition~\ref{prop:topg}).  It follows that $\Delta_{g,n}$ has the rational homotopy type of a $(n + 2g-4  +  \delta_{0,n})$-connected complex.  However, we do not know whether $\Dgn$ is more than $(n-3)$-connected, for $g > 2$.

We now discuss two applications to the topology of the complex moduli space $\cM_{g,n}$, in the forms of Theorems~\ref{thm:cohomology} and~\ref{thm:dualbasis} below.  First, as noted above, the reduced rational homology of $\Dgn$ is canonically identified with the top weight cohomology of $\cM_{g,n}$, so Theorem~\ref{thm:boundary} determines the top weight cohomology of $\cM_{1,n}$ as a $\Q$-vector space.  With a little more care, we can work $S_n$-equivariantly and compute this top weight cohomology as a representation of the symmetric group, as follows.

\begin{theorem} \label{thm:cohomology}
For each $n\ge 1$, the top weight cohomology of $\cM_{1,n}$ is
\[
\Gr_{2n}^W H^i(\cM_{1,n}; \Q) \cong \left \{ \begin{array}{ll} \Q^{(n-1)!/2} & \mbox{ \ \ for $n \geq 3$ and $i = n$,} \\	
 												0 & \mbox{ \ \ otherwise.} \end{array} \right.
\]
Moreover, for each $n\ge 3$, the representation of $S_n$ on $\Gr^W_{2n} H^n(\mathcal{M}_{1,n}; \Q)$ induced by permuting marked points is
\[
\mathrm{Ind}_{D_n,\phi}^{S_n} \, \mathrm{Res}^{S_n}_{D_n,\psi} \, \mathrm{sgn}.
\] 
\noindent Here $\phi\colon D_n \rightarrow S_n$ is the dihedral group of order $2n$ acting on the vertices of an $n$-gon, $\psi\colon D_n \rightarrow S_n$ is the action of the dihedral group on the edges of the $n$-gon, and $\mathrm{sgn}$ denotes the sign representation of $S_n$.   
 \end{theorem}
   
\noindent Note that the signs of these two permutation actions of the dihedral group are different when $n$ is even.  Reflecting a square across a diagonal, for instance, exchanges one pair of vertices and two pairs of edges.

\begin{remark} \label{rem:genus0}
An analogous result is known for $g = 0$. By \cite{Vogtmann90} and the identification of the top weight cohomology of $\cM_{0,n}$ with the reduced rational homology of $\cM_{0,n}$, the top weight cohomology $\Gr_{2n-6}^W H^{*}(\cM_{0,n}; \Q)$ has rank $(n-2)!$, concentrated in degree $n -3$, for $n \geq 3$, and its character as a representation of $S_n$ is computed in \cite{RobinsonWhitehouse96}.  
\end{remark}

\begin{remark}
The fact that the top weight cohomology of $\cM_{1,n}$ is supported in degree $n$ can also be seen without tropical methods, as follows.  The rational cohomology of a smooth Deligne--Mumford stack agrees with that of its coarse moduli space, and the coarse space $M_{1,n}$ is affine.  To see this, note that $M_{1,1}$ is affine, and the forgetful map $M_{g,n+1} \rightarrow M_{g,n}$ is an affine morphism for $n \geq 1$.  It follows that $M_{1,n}$ has the homotopy type of a $n$-dimensional CW-complex, by \cite{AndreottiFrankel59, Karcjauskas77}, and hence $H^*(\cM_{1,n}; \Q)$ is supported in degrees less than or equal to $n$. The weights on $H^k$ are always between $0$ and $2k$, so the top weight $2n$ can appear only in degree $n$.  

Getzler has calculated an expression for the $S_n$-equivariant Serre characteristic of $\cM_{1,n}$ \cite[(5.6)]{Getzler99}.  Since the top weight cohomology is supported in a single degree, it is determined as a representation by this equivariant Serre characteristic.  Nevertheless, we do not know how to deduce the statement of Theorem~\ref{thm:cohomology} from Getzler's formula.  

On the other hand, Petersen explains that it is possible to adapt the methods from \cite{Petersen14} to recover the fact that the top weight cohomology of $\cM_{1,n}$ has rank $(n-1)!/2$.  He uses the Leray spectral sequence for $\cM_{1,n} \rightarrow \cM_{1,1}$ to determine the weight zero compactly supported cohomology dual to top weight cohomology, decomposing $R^q f_{!} \Q$ as a sum of local systems associated to the symmetric power representations of $\SL_2$, computing which summands appear without Tate twists, and determining the weight zero cohomology of each of these local systems using the Eichler--Shimura isomorphism to compare with spaces of Eisenstein series (personal communication, October 2015).
\end{remark}

\begin{remark}
It is tempting to try to fit the sequence of $S_n$-representations given by $\Gr^W_{2n} H^n(\mathcal{M}_{1,n};\Q)$ into the framework of \emph{representation stability} from \cite{ChurchFarb13}.  However, despite the uniform description of the representations given by Theorem~\ref{thm:cohomology} there does not seem to be any natural $S_n$-equivariant map $H^n(\mathcal{M}_{1,n};\Q) \to H^{n+1}(\mathcal{M}_{1,n+1};\Q)$.  Moreover, as shown in \cite{ChurchFarb13}, representation stability implies polynomial dimension growth, whereas the dimension of $\Gr_{2n}^W H^n(\cM_{1,n}; \Q)$ grows super exponentially with $n$.  

If we fix both $k$ and $g$, then $H_k(\Dgn, \Q)$ vanishes for $n \gg 0$, by Theorem~\ref{thm:connected}.  This vanishing might be seen as an analogue of stabilization with respect to marked points for homology of the pure mapping class group, which says that, for fixed $g$ and $k$, the sequence $H_k(\cM_{g,n}, \Q)$ is representation stable \cite{JimenezRolland11}.\end{remark}

\begin{remark}
The weight filtration is one part of Deligne's mixed Hodge structure on the cohomology of an algebraic variety (or Deligne--Mumford stack).  The full mixed Hodge structure on the cohomology of $\cM_{g,n}$ has been computed in only a few positive genus cases, using stratifications and either exact sequences in cohomology \cite{Looijenga93} or point counting over finite fields \cite{Tommasi07, BergstromTommasi07}.  In this paper, we do not compute any full mixed Hodge structures; instead, we focus on one graded piece of the weight filtration that can be isolated and understood separately.
\end{remark}

The last theorem that we state in the introduction refines Theorem~\ref{thm:cohomology} by producing explicit cycles in $\cM_{1,n}$ whose classes form a dual basis for $\Gr_{2n}^W H^n(\cM_{1,n}; \Q)$, for $n \geq 3$.  To describe these cycles, consider the $(n-1)!/2$ isomorphism classes of marked cycle graphs of length $n$, in which the vertices are labeled by distinct elements of $\{1, \ldots, n\}$.  Each such graph is the dual graph of a loop of $n$ smooth rational curves, with one marked point on each component.  An example is shown in Figure~\ref{f:loop}.  

%%%%%%%%%%%%%%%%%%%%%%%%%%%%%%%%%%%%
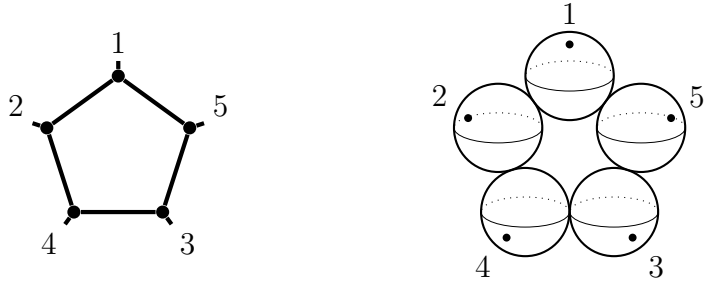
\begin{figure}
\begin{tikzpicture}[my_node/.style={fill, circle, inner sep=1.75pt}, scale=1]
\def\R{.587} %radius
\def\H{.2} %height of dotted ellipse
\def\Ratio{1.42} %pos of marked pt
\def\M{1.2}
% I
\begin{scope}[scale=1]
\node[my_node, label={[label distance=2pt]90:$1$}] (1) at (0,1){};
\node[my_node, label={[label distance=2pt]5:$5$}] (5) at (.95,.31){};
\node[my_node, label={[label distance=2pt]-60:$3$}] (3) at (0.59,-.81){};
\node[my_node, label={[label distance=2pt]240:$4$}] (4) at (-.59,-.81){};
\node[my_node, label={[label distance=2pt]175:$2$}] (2) at (-.95,.31){};
\draw[ultra thick] (1)--(5)--(3)--(4)--(2)--(1);
\draw[ultra thick] (1)--(0, 1*\M);
\draw[ultra thick] (5)--(.95*\M, .31*\M);
\draw[ultra thick] (3)--(.59*\M, -.81*\M);
\draw[ultra thick] (4)--(-.59*\M, -.81*\M);
\draw[ultra thick] (2)--(-.95*\M, .31*\M);
\end{scope}
% II
\begin{scope}[shift= {(6,0)}]
\begin{scope}[shift = {(.95,.31)}, scale=1]
\draw[thick] (0,0) circle (\R);
\draw[thin, dotted] (\R,0) arc [x radius =\R, y radius=\H, start angle = 0, end angle = 180];
\draw[ultra thin]  (-\R,0) arc [x radius =\R, y radius=\H, start angle = 180, end angle = 360];
\end{scope}
\begin{scope}[shift = {(-.95,.31)}, scale=1]
\draw[thick] (0,0) circle (\R);
\draw[thin, dotted] (\R,0) arc [x radius =\R, y radius=\H, start angle = 0, end angle = 180];
\draw[ultra thin]  (-\R,0) arc [x radius =\R, y radius=\H, start angle = 180, end angle = 360];
\end{scope}
\begin{scope}[shift = {(0,1)}, scale=1]
\draw[thick] (0,0) circle (\R);
\draw[thin, dotted] (\R,0) arc [x radius =\R, y radius=\H, start angle = 0, end angle = 180];
\draw[ultra thin]  (-\R,0) arc [x radius =\R, y radius=\H, start angle = 180, end angle = 360];
\end{scope}
\begin{scope}[shift = {(0.59,-0.81)}, scale=1]
\draw[thick] (0,0) circle (\R);
\draw[thin, dotted] (\R,0) arc [x radius =\R, y radius=\H, start angle = 0, end angle = 180];
\draw[ultra thin]  (-\R,0) arc [x radius =\R, y radius=\H, start angle = 180, end angle = 360];
\end{scope}
\begin{scope}[shift = {(-.59,-.81)}, scale=1]
\draw[thick] (0,0) circle (\R);
\draw[thin, dotted] (\R,0) arc [x radius =\R, y radius=\H, start angle = 0, end angle = 180];
\draw[ultra thin]  (-\R,0) arc [x radius =\R, y radius=\H, start angle = 180, end angle = 360];
\end{scope}
\node[my_dot, label={[label distance=2pt]90:$1$}] at (\Ratio*0,\Ratio*1){};
\node[my_dot, label={[label distance=1pt]5:$5$}] at (\Ratio*.95,\Ratio*.31){};
\node[my_dot, label={[label distance=2pt]-60:$3$}] at (\Ratio*0.59,\Ratio*-.81){};
\node[my_dot, label={[label distance=2pt]240:$4$}] at (\Ratio*-0.59,\Ratio*-.81){};
\node[my_dot, label={[label distance=2pt]175:$2$}] at (\Ratio*-.95,\Ratio*.31){};
\end{scope}
\end{tikzpicture}
\caption{A marked graph $\Gmw$ indexing a top-dimensional cell of the boundary complex of $\ov{\mathcal{M}}_{1,5}$, and the stable curve of $\ov{\mathcal{M}}_{1,5}$ dual to it.}
\label{f:loop}
\end{figure}
%%%%%%%%%%%%%%%%%%%%%%%%%%%%%%%%%%%%%%%%%%%%%%%%%%%%%%%%

Notice that these curves are stable, with no automorphisms that preserve the marked points, and the $(n-1)!/2$ geometric points in $\ov \cM_{1,n}$ corresponding to them are a proper subset of the zero-dimensional boundary strata.  In a neighborhood of the $0$-stratum corresponding to a graph $\Gmw$, the pair $(\ov \cM_{1,n},  \partial \ov \cM_{1,n})$ looks locally like $\C^n$ together with its coordinate hyperplanes, so we can embed a real $n$-torus near this point and disjoint from the boundary as the locus
\[
T_{\Gmw, \epsilon} = \{ (z_1, \ldots, z_n) \in \C^n \ | \ |z_i| = \epsilon \mbox{ for all } i \},
\]
for some small positive $\epsilon$.  The homology class of $T_{\Gmw, \epsilon}$ is independent of $\epsilon$ and the choice of coordinates (only its orientation depends on the ordering of the coordinates, which can be fixed in terms of the cyclic ordering of the marked points), so we denote it by $[\mathfrak{t}_{\Gmw}]$.  We remark that the class $[\mathfrak{t}_{\Gmw}]$ is an abelian cycle, in the sense of \cite{BrendleFarb07, ChurchFarb12}, for the pure mapping class group of a genus 1 surface with $n$ marked punctures, coming from an explicit $n$-tuple of commuting Dehn twists.  See Proposition~\ref{prop:abelian}.

As we explain in \S\ref{sec:stackboundary} and the appendix, standard computations with the logarithmic de Rham complex and residues show that the classes $[\mathfrak{t}_\Gmw]$ are orthogonal to $W_{2n-1} H^n (\cM_{1,n}; \Q)$, and hence pair naturally with classes in $\Gr_{2n}^W H^n(\cM_{1,n}; \Q)$.  We then show:

\begin{theorem} \label{thm:dualbasis}
For each $n\ge 3$, the classes $[\mathfrak{t}_\Gmw] \in H_n(\cM_{1,n}; \Q)$ indexed by cycles with $n$ vertices labeled by distinct elements of $\{1, \ldots, n\}$ form a dual basis to $\Gr_{2n}^W H^n(\cM_{1,n}; \Q)$.
\end{theorem}

\noindent The proof of Theorem~\ref{thm:dualbasis} consists of unraveling the identification of the reduced homology of the boundary complex with top weight cohomology using the weight filtration on the logarithmic de Rham complex, the spectral sequence associated to this filtration, and the theory of Poincar\'e residues, along the lines followed by Hacking in his work on homology of tropical varieties \cite{Hacking08}.  Similar arguments characterize the relations among abelian cycles of middle homological degree in $\cM_{g,n}$ in terms of coboundaries on the tropical moduli space.  See Theorem~\ref{thm:abelianrelations}.  This line of argument is standard for varieties and the generalization to smooth Deligne--Mumford stacks is straightforward.  However, lacking a suitable reference, we provide details in the appendix.  

\medskip

\noindent \textbf{Structure of the paper.}  A reader who is interested only in the topology of $\Dgn$ and relations to outer space should read \S\ref{sec:graphs} and may proceed directly from there to the proofs of Theorems~\ref{thm:contractible}, \ref{thm:boundary}, and \ref{thm:connected} in \S\ref{sec:contractibility}-\ref{sec:proofs}, skipping \S\ref{sec:cellular}-\S\ref{sec:abcycle} and the appendix.  (The combinatorial topological constructions in \S\ref{sec:cellular} may also be of interest to such readers, but are not required for the proofs of these three theorems.) See also the proof of Theorem~\ref{thm:cohomology} in \S\ref{sec:proofs} for a computation of the rational homology of $\Dn$ as a representation of $S_n$.

The basic background material on boundary complexes and top weight cohomology reviewed in \S\ref{sec:stackboundary} allows us to interpret topological invariants of $\Dgn$ as algebraic invariants of $\cM_{g,n}$. This is essential for the proof of Theorem~\ref{thm:cohomology}.  

The more refined statements involving cellular chains and cochains, residue integrals of logarithmic forms, and torus classes, presented in \S\ref{sec:cellular}-\ref{sec:log} and the appendix, are key ingredients in the proof of Theorem~\ref{thm:dualbasis}.  In \S\ref{sec:abcycle}, we explain the identification of torus classes in $\cM_{g,n}$ with abelian cycles for the pure mapping class group. This applies, in particular, to the explicit $(n-1)!/2$ torus classes that give a dual basis to the top weight cohomology of $\cM_{1,n}$, for $n \geq 3$, and allows us to describe the relations among middle dimensional abelian cycles in terms of coboundaries on $\Dgn$ in Theorem~\ref{thm:abelianrelations}.

\bigskip

\noindent \textbf{Acknowledgments.}  We thank D.~Abramovich, E.~Getzler, R.~Hain, M.~Kahle, L.~Migliorini, D.~Petersen, N.~Salter, C.~Simpson, O.~Tommasi, and R.~Vakil for helpful conversations related to this work.  MC was supported by NSF DMS-1204278.  SG was supported by NSF DMS-1405001.  SP was supported by NSF CAREER DMS-1149054 and is grateful to the Institute for Advanced Study for ideal working conditions in Spring 2015.

\section{Graphs, tropical curves, and moduli}  \label{sec:graphs}

In this section, we briefly recall the construction of $\Dgn$ as a moduli space for tropical curves, which are marked weighted graphs with a length assigned to each edge.  We also discuss its relationship to Culler--Vogtmann outer space and homological stability in~\S\ref{ss:outerspace}.

\subsection{Marked weighted graphs and tropical curves} \label{subsec:J-g-n}

Let $G$ be a finite graph, possibly with loops and parallel edges.  All graphs in this paper will be connected.  Write $V(G)$ and $E(G)$ for the vertex set and edge set, respectively, of $G$.
A {\em weighted graph} is a connected graph $G$ together with a function $w\col V(G)\rightarrow \Z_{\ge 0}$, called the {\em weight function}.  The {\em genus} of $(G,w)$ is
\[
g(G,w)= b_1(G) + \sum_{v\in V(G)}\! w(v),
\]
where $b_1(G) = |E(G)|-|V(G)|+1$ is the first Betti number of $G$.  The {\em core} of a weighted graph $(G,w)$ is the smallest connected subgraph of $G$ that contains all cycles of $G$ and all vertices of positive weight. So the core is nonempty if and only if $g(G,w) > 0$.

An {\em $n$-marking} on $G$ is a map $m\col\{1,\ldots,n\}\rightarrow V(G)$.  In pictures, we may think of the marking as a set of $n$ labeled half-edges attached to the vertices of $G$.

An {\em $n$-marked weighted graph} is a triple $\Gmw=(G,m,w)$, where $(G,w)$ is a weighted graph and $m$ is an $n$-marking. The {\em valence} of a vertex $v$ in a marked weighted graph, denoted $\val(v)$, is the number of half-edges of $G$ incident to $v$ plus the number of marked points at $v$.  In other words, a loop edge based at $v$ counts twice towards $\val(v)$, once for each end, an ordinary edge counts once, and a marked point counts once. We say that $\Gmw $ is {\em stable} if for every $v\in V(G)$, $$2w(v) -2 + \val(v) > 0.$$  Equivalently, $\Gmw$ is stable if and only if every vertex of weight 0 has valence at least 3, and every vertex of weight 1 has  valence at least 1.

\subsection{The category $J_{g,n}$}  The connected stable $n$-marked graphs of genus $g$ form the objects of a \emph{category} which we denote $J_{g,n}$.  The morphisms in this category are compositions of contractions of edges $G\rightarrow G/e$ and isomorphisms $G{\rightarrow}G'.$  Because we shall need to take colimits of functors out of $J_{g,n}$ in what follows, we now give a formal and precise definition of $J_{g,n}$. 

Formally, then, a graph $G$ is a finite set $X(G) = V(G) \sqcup H(G)$ (of ``vertices'' and ``half-edges''), together with two functions $s_G, r_G\colon X(G) \to X(G)$ satisfying $s_G^2 = \mathrm{id}$ and $r_G^2 = r_G$ and that $$\{x \in X(G) \mid r_G(x) = x\} = \{x \in X(G) \mid s_G(x) = x\} = V(G).$$ Informally: $s_G$ sends a half-edge to its other half, while $r_G$ sends a half-edge to its incident vertex.  We let $E(G) = H(G)/(x \sim s_G(x))$ be the set of edges.  The definition of $n$-marking, weights, genus, and stability is as before.  

The objects of the category $J_{g,n}$ are all connected stable $n$-marked graphs of genus $g$.
For an object $\Gmw = (G,m,w)$ we we shall write $V(\Gmw)$ for $V(G)$ and similarly for $H(\Gmw)$, $E(\Gmw)$, $X(\Gmw)$, $s_\Gmw$ and $r_\Gmw$.  Then a morphism $\Gmw \to \Gmw'$ is a function $f\colon X(\Gmw) \to X(\Gmw')$ such that $$f \circ r_\Gmw = r_{\Gmw'} \circ f \text{ and }f \circ s_\Gmw = s_{\Gmw'} \circ f$$ subject to the following three requirements:  

\begin{itemize}
\item Noting first that $f(V(\Gmw))\subseteq V(\Gmw')$, we require $f\circ m = m'$, where $m$ and $m'$ are the respective marking functions of $\Gmw$ and $\Gmw'$. 
\item Each $e \in H(\Gmw')$ determines the subset $f^{-1}(e) \subset X(\Gmw)$ and we require that it consists of precisely one element (which will then automatically be in $H(\Gmw)$).  
\item Each $v \in V(\Gmw')$ determines a subset $S_v = f^{-1}(v) \subset X(\Gmw)$ and $\mathbf{S}_v = (S_v,r\vert_{S_v}, s\vert_{S_v})$ is a graph; we require that it be connected and have $g(\mathbf{S}_v,w\vert_{\mathbf{S}_v}) = w(v)$.  
\end{itemize}

\noindent Composition of morphisms $\Gmw \to \Gmw' \to \Gmw''$ in $J_{g,n}$ is given by the corresponding composition $X(\Gmw) \to X(\Gmw') \to X(\Gmw'')$ in the category of sets.

Our definition of graphs and the morphisms between them agrees, in essence, with the definitions in \cite[X.2]{ACG11} and \cite[\S3.2]{acp}.  One feature worth noting is the involution associated to each loop edge; if $e$ is a loop in $G$, consisting of half-edges $h$ and $h'$, then flipping the loop (i.e., exchanging $h$ and $h'$) gives a nontrivial automorphism of $G$.

\begin{remark}
We also note that any morphism $\Gmw \to \Gmw'$ can be alternatively described as an isomorphism following a finite sequence of \emph{edge collapses}: if $e \in E(\Gmw)$ there is a morphism $\Gmw \to \Gmw/e$ where $\Gmw/e$ is the marked weighted graph obtained from $\Gmw$ by collapsing $e$ together with its two endpoints to a single vertex $[e] \in \Gmw/e$.  If $e$ is not a loop, the weight of $[e]$ is the sum of the weights of the endpoints of $e$ and if $e$ is a loop the weight of $[e]$ is one more than the old weight of the end-point of $e$.  If $S = \{e_1, \dots, e_k\} \subset E(\Gmw)$ there are iterated edge collapses $\Gmw \to \Gmw/e_1 \to (\Gmw/e_1)/e_2 \to \dots$ and any morphism $\Gmw \to \Gmw'$ can be written as such an iteration followed by an isomorphism from the resulting quotient of $\Gmw$ to $\Gmw'$.
\end{remark}

Let us also point out that a morphism $f: \Gmw \to \Gmw'$ determines an injective function $H(\Gmw') \to H(\Gmw)$, sending $e' \in H(\Gmw')$ to the unique element $e \in H(\Gmw)$ with $f^{-1}(e') = \{e\}$.  We shall write $f^{-1}\colon H(\Gmw') \to H(\Gmw)$ for this map and also for the induced injective map $E(\Gmw') \to E(\Gmw)$.  This construction determines functors $H,E: J_{g,n}^\mathrm{op} \to \mathrm{Sets}$.

Finally, we shall say that $\Gmw$ and $\Gmw'$ have the same \emph{combinatorial type} if they are isomorphic in $J_{g,n}$.  In order to get a \emph{small} category $J_{g,n}$ we shall tacitly pick one object in each isomorphism class and pass to the full subcategory on those objects.  In fact there are only finitely many isomorphism classes of objects in $J_{g,n}$, since any object has at most $6g-6+2n$ half-edges and $2g-2+n$ vertices and for each possible set of vertices and half-edges there are finitely many ways of gluing them to a graph, and finitely many possibilities for the $n$-marking and weight function.

\subsection{Moduli space of tropical curves}  

We now recall the construction of moduli spaces of stable tropical curves, as the colimit of a diagram of cones parametrizing possible lengths of edges for each fixed combinatorial type.  The construction follows \cite{BrannettiMeloViviani11, Caporaso13}.  

Fix integers $g,n\ge 0$ with $2g-2+n > 0$. A \emph{length function} on $\Gmw = (G,m,w) \in J_{g,n}$ is an element $\ell \in \R_{>0}^{E(\Gmw)}$, and we shall think geometrically of $\ell(e)$ as the \emph{length} of the edge $e \in E(\Gmw)$.  An $n$-marked genus $g$ \emph{tropical curve} is then a pair $\Gamma = (\Gmw,\ell)$ with $\Gmw \in J_{g,n}$ and $\ell \in \R_{>0}^{E(\Gmw)}$, and we shall say that $(\Gmw,\ell)$ is \emph{isometric} to $(\Gmw',\ell')$ if there exists an isomorphism $\phi: \Gmw \to \Gmw'$ in $J_{g,n}$ such that $\ell' = \ell \circ \phi^{-1}: E(\Gmw') \to \R_{>0}$.  The \emph{volume} of $(\Gmw,\ell)$ is $\sum_{e \in E(\Gmw)} \ell(e) \in \R_{> 0}$.

We can now describe the underlying set of the topological space $\Delta_{g,n}$, which is the main object of study in this paper. It is the set of isometry classes of $n$-marked genus $g$ tropical curves of volume 1.  We proceed to describe its topology and further structure as a closed subspace of the moduli space of tropical curves.

\begin{definition}\label{definition:mgn}
Fix $g,n\ge 0$ with $2g-2+n > 0$.  For each object $\Gmw \in J_{g,n}$ define the topological space
$$\sigma(\Gmw)= \R_{\geq 0}^{E(\Gmw)}=\{\ell: E(\Gmw) \to \R_{\geq 0}\}.$$  For a morphism $f\colon  \Gmw \to \Gmw'$ define the continuous map $\sigma f \colon \sigma(\Gmw') \to \sigma(\Gmw)$ by $(\sigma f)(\ell') = \ell\colon E(\Gmw) \to \R_{\geq 0}$, where $\ell$ is given by
  
$$\ell(e) = \begin{cases} 
\ell'(e') & \text{ if $f$ sends $e$ to $e'\in E(\Gmw')$},\\
0 & \text{  if $f$ collapses $e$ to a vertex}.
\end{cases}$$
This defines a functor $\sigma: J_{g,n}^\mathrm{op} \to \mathrm{Spaces}$ and the topological space $\Mgn$ is defined to be the colimit of this functor.
\end{definition}

In other words, the topological space $\Mgn$ is obtained as follows. For each morphism $f\colon \Gmw\to\Gmw'$, consider the map $L_f\colon \sigma(\Gmw') \to \sigma(\Gmw)$ that sends $\ell'\colon E(\Gmw')\to \R_{>0}$ to the length function $\ell\colon E(\Gmw)\to \R_{>0}$ obtained from $\ell'$ by extending it to be 0 on all edges of $\Gmw$ that are collapsed by $f$.  So $L_f$ linearly identifies $\sigma(\Gmw')$ with some face of $\sigma(\Gmw)$, possibly $\sigma(\Gmw)$ itself.  Then
$$\Mgn = \left(\coprod \sigma(\Gmw)\right) \Big / \{\ell' \sim L_f(\ell')\},$$
where the equivalence runs over all morphisms $f\colon \Gmw\to\Gmw'$ and all $\ell'\in \sigma(\Gmw')$.  

In fact, we shall sometimes regard $\Mgn$ as slightly more than just a topological space: $\Mgn$ is
an example of a {\em generalized cone complex}, as defined in \cite[\S2]{acp}.  Formally, the projections
$\sigma(\Gmw) = \R_{\geq 0}^{E(\Gmw)} \to \R$, one for each element of $E(\Gmw)$, give each $\sigma(\Gmw)$ the structure of a \emph{cone}.

The \emph{volume} defines a function $v: \sigma(\Gmw) \to \R_{\geq 0}$, given explicitly as $v(\ell) = \sum_{e \in E(\Gmw)} \ell(e)$, and for any morphism $\Gmw \to \Gmw'$ in $J_{g,n}$ the induced map
$\sigma(\Gmw) \to \sigma(\Gmw')$ preserves volume.  Hence there is an induced map $v: \Mgn \to \R_{\geq 0}$, and there is a unique element in $\Mgn$ with volume 0 which we shall denote $\bullet_{g,n}$.  The underlying graph of $\bullet_{g,n}$ consists of a single vertex with weight $g$ and containing all $n$ marked points.

\begin{definition}\label{def:dgn}
We let $\Dgn$ be the subspace of $\Mgn$ parametrizing curves of volume 1, i.e.,\ the inverse image of $1 \in \R$ under $v: \Mgn \to \R$.
\end{definition}

\noindent Thus $\Dgn$ is homeomorphic to the link of $\Mgn$ at the cone point $\bullet_{g,n}$.  Moreover, it inherits the structure of a generalized $\Delta$-complex, as we shall define in \S\ref{sec:cellular}, from the generalized cone complex structure on $\Mgn$.  See Example~\ref{ex:Dgn} and \S\ref{sec:cones}.
 
\subsection{Relationship with outer space}\label{ss:outerspace}

It is natural to ask whether the homology of $\Dgn$ can be related to known instances of {\em homological stability} for the complex moduli space of curves $\cM_{g,n}$ and for the free group $F_g$.  Here, we pause for a moment to discuss briefly the reasons that the tropical moduli space $\Dgn$ relates to both $\cM_{g,n}$ and $F_g$.

{Homological stability} has been an important point of view in the understanding of $\cM_{g,n}$; we are referring to the fact that the cohomology group $H^k(\cM_{g,n}; \Q)$ is independent of $g$ as long as $g \geq 3k/2 + 1$ \cite{Harer, Ivanov93, Boldsen12}.  The structure of the rational cohomology in this stable range was conjectured by Mumford and proved by Madsen and Weiss \cite{MadsenWeiss07}.  There are certain \emph{tautological classes} $\kappa_i \in H^{2i}(\cM_{g,n})$ and $\psi_i \in H^2(\cM_{g,n})$, and the induced map
\begin{equation*}
\Q[\kappa_1, \kappa_2, \dots] \otimes \Q[\psi_1, \dots, \psi_n] \to H^*(\cM_{g,n};\Q) 
\end{equation*}
is an isomorphism in the stable range.

A very similar homological stability phenomenon happens for \emph{automorphisms of free groups}.  If $F_g$ is the free group on $g$ generators and $\mathrm{Aut}(F_g)$ is its automorphism group, then Hatcher and Vogtmann \cite{HatcherVogtmann98} proved that the group cohomology of $H^k(\mathrm{Aut}(F_g))$ is independent of $g$ as long as $g \gg k$.  Galatius proved the analogue of the theorem of Madsen and Weiss for these groups: the rational cohomology $H^k(B\mathrm{Aut}(F_g);\Q)$ vanishes for $g \gg k$ \cite{Galatius11}.

The tropical moduli space $\Dgn$ is closely related to \emph{both} of these objects.  On the one hand its reduced rational homology is identified with the top weight cohomology of $\cM_{g,n}$.  On the other hand it is also closely related to $\mathrm{Aut}(F_g)$, as we shall now briefly explain. One of the initial motivations for this paper was to use the tropical moduli space to provide a direct link between moduli space of curves and automorphism groups of free groups.

The tropical moduli space has an open subset $\Dgn^{\mathrm{pure}}\subset \Dgn$ parametrizing tropical curves with zero vertex weights.  In fact $\Delta_{g,1}^\mathrm{pure}$ is a rational model for $B\mathrm{Aut}(F_g)$, in the sense that there is a map
\begin{equation*}
B\mathrm{Aut}(F_g) \to \Delta_{g,1}^\mathrm{pure}
\end{equation*}
inducing an isomorphism in rational cohomology.  This rational homology equivalence is deduced from the contractibility of Culler and Vogtmann's \emph{outer space} \cite{CullerVogtmann86, HatcherVogtmann98}.  (Recall that $B\mathrm{Aut}(F_g)$ denotes the \emph{classifying space} of the discrete group $\mathrm{Aut}(F_g)$.  It is a $K(\pi,1)$ space whose singular cohomology is isomorphic to the group cohomology of $\mathrm{Aut}(F_g)$.)

The inclusion $\Dgn^\mathrm{pure} \subset \Dgn$ at first sight seems to give the desired direct link between automorphism groups of free groups and moduli space of curves.  For $n=1$, this inclusion induces a map in cohomology
\begin{equation}\label{eq:3}
\Gr_{6g-4}^W H^*(\cM_{g,1};\Q) \to H^*(B\mathrm{Aut}(F_g);\Q).
\end{equation}
It is tempting to try to form some kind of direct limit of~\eqref{eq:3} as $g \to \infty$.  For $n=1$, there is indeed a map $\Delta_{g,1} \to \Delta_{g+1,1}$, given by taking wedge sum with a circle at the marked point.  This map fits with the stabilization map for $B\mathrm{Aut}(F_g)$ into a commutative diagram
\begin{equation}\label{eq:2}
  \begin{aligned}
    \xymatrix{
      {B\mathrm{Aut}(F_g)} \ar[d] \ar[r]^-{\simeq_\Q} &
      {\Delta_{g,1}^\mathrm{pure}} \ar[d] \ar[r] &
      {\Delta_{g,1}} \ar[d]\\
      {B\mathrm{Aut}(F_{g+1})} \ar[r]_-{\simeq_\Q} &
      {\Delta_{g+1,1}^\mathrm{pure}}
      \ar[r] &{\Delta_{g+1,1}}.
    }
  \end{aligned}
\end{equation}
For the outer automorphism group $\mathrm{Out}(F_g)$ there is a similar comparison diagram
\begin{equation*}
  \xymatrix{
    {B\mathrm{Aut}(F_g)} \ar[d] \ar[r]^-{\simeq_\Q} &
    {\Delta_{g,1}^\mathrm{pure}} \ar[d] \ar[r] &
    {\Delta_{g,1}} \ar[d]\\
    {B\mathrm{Out}(F_{g})} \ar[r]_-{\simeq_\Q} &
    {\Delta_{g,0}^\mathrm{pure}}
    \ar[r] &{\Delta_{g,0}},
  }
\end{equation*}
and for $n > 1$ there is a similar diagram with $\Dgn^\mathrm{pure} \subset \Dgn$ compared with the groups $A_{g,n}$ from \cite{HatcherWahl10}.

In light of \cite{MadsenWeiss07} and \cite{Galatius11} it is tempting to also ask for the limiting cohomology of $\Delta_{g,1}$ as $g \to \infty$.  However, this limit behaves very differently from the corresponding limits for $B\mathrm{Aut}(F_g)$ and $\cM_{g,n}$.  In fact, the stabilization maps $\Delta_{g,1} \to \Delta_{g+1,1}$ are easily seen to be null homotopic and hence the limiting cohomology vanishes. The main result of \cite{Galatius11} implies that the limiting rational cohomology of the subspaces $\Delta_{g,1}^\mathrm{pure}$ also vanishes, but the proof is rather involved.  To see that the stabilization maps $\Delta_{g,1} \to \Delta_{g+1,1}$ in~(\ref{eq:2}) are null homotopic, first recall that they send a tropical curve $\Gamma \in \Delta_{g,1}$ to $\Gamma \vee S^1$, appropriately normalized (e.g.\ multiply all edge lengths in $\Gamma$ by $\frac12$ and give the $S^1$ length $\frac12$).  Continuously shrinking the volume of $\Gamma \subset \Gamma \vee S^1$ and increasing that of $S^1$ defines a homotopy from the stabilization map to the constant map $\Delta_{g,1} \to \Delta_{g+1,1}$ which sends any weighted tropical curve $\Gamma$ to a circle of length 1, whose basepoint has weight $g$.

\section{Generalizations of $\Delta$-complexes}  \label{sec:cellular}

In this section, we review and develop some basic notions in combinatorial topology with a view towards understanding $\Dgn$ and, more generally, the dual complex of any normal crossings divisor in a smooth Deligne--Mumford (DM) stack, as an object that we call a \emph{generalized $\Delta$-complex.}  We also present a theory of cellular chains and cochains in this framework, which will be especially useful for applications to $\cM_{g,n}$.  For instance, in \S\ref{sec:residues}, we will explain how integrals of Poincar\'e residues of logarithmic differential forms on $\cM_{g,n}$ are naturally interpreted as cellular chains on $\Dgn$.

Recall that a $\Delta$-complex can be defined as a topological space $X$ equipped with a collection of maps $\Delta^p \to X$ for varying $p$, where $\Delta^p$ denotes the standard $p$-simplex, satisfying a collection of axioms which guarantee that $X$ can be obtained by gluing these simplices along faces.  The vertices of $\Delta^p$ come with a total ordering, and the gluing maps are required to preserve this ordering.  See \cite[\S2.1]{Hatcher02} for a precise definition along these lines.  One can also give a combinatorial definition of $\Delta$-complexes, in terms of discrete data recording how the simplices are to be glued.  We recall such a definition in \S\ref{sec:Delta-complexes}.

We discuss two generalizations of $\Delta$-complexes.  The first, which we call \emph{unordered $\Delta$-complexes}, allows distinct simplices to be glued along faces using maps that do not necessarily preserve the ordering of the vertices.  The second, which we call \emph{generalized $\Delta$-complexes}, also allows a simplex to be glued to itself along maps which permute its vertices.

Before delving into the definitions, let us point out that the association taking a non-negative integer $p$ to the standard simplex $\Delta^p$ can be promoted to a functor from finite sets to topological spaces; for a finite set $S$ define $\Delta^S = \{a: S \to [0,\infty) \mid \sum a(s) = 1\}$ in the Euclidean topology and for any map of finite sets $\theta: S \to T$, define $\theta_*: \Delta^S \to \Delta^T$ by
\begin{equation*}
  (\theta_* a)(t) = \sum_{\theta(s) = t} a(s).
\end{equation*}
The usual $p$-simplex is recovered as $\Delta^p = \Delta^{[p]}$ with $[p] = \{0, \dots, p\}$.

\subsection{Ordinary $\Delta$-complexes}
\label{sec:Delta-complexes}

Let $\Delta_\inj$ be the category with one object $[p] = \{0, \dots, p\}$ for each integer $p \geq 0$, in which the morphisms $[p] \to [q]$ are the order preserving injective maps.  We shall take the following as the official definition of a $\Delta$-complex.

\begin{definition}
A $\Delta$-complex is a presheaf on $\Delta_\inj$.
\end{definition}

\noindent In other words, a $\Delta$-complex is a contravariant functor from $\Delta_\inj$ to sets.  When $X$ is a $\Delta$-complex and $\theta: [p] \rightarrow [q]$ is a morphism in $\Delta_\inj$, we write $\theta^* : X([q]) \rightarrow X([p])$ for the induced map $X(\theta)$.

The \emph{geometric realization} of a $\Delta$-complex $X$ is the topological space
\begin{equation}\label{eq:realization}
  |X| = \Big(\coprod_{p = 0}^\infty X([p]) \times \Delta^p\Big) \big / \sim,
\end{equation}
where $\sim$ is the equivalence relation generated by $(x,\theta_* a) \sim (\theta^* x, a)$ for $x \in X([q])$, $\theta: [p] \to [q]$ in $\Delta_\inj$, and $a \in \Delta^p$.  Each element $x \in X([p])$ determines a map of topological spaces $x: \Delta^p \to |X|$, and these maps in turn define the cells in a CW structure on $|X|$.  Many sources, including \cite{Hatcher02}, take the official definition to be a topological space equipped with a set of maps from simplices satisfying certain axioms.  The intuition is that $X([p])$ has one element for each $p$-simplex and the functoriality determines how the simplices are glued together.  In any case, these two different approaches produce equivalent categories.

As is customary, we shall occasionally write $X_p = X([p])$.  Let us also write $\delta^i: [p-1] \to [p]$ for the unique injective order-preserving map whose image does not contain $i$, and $d_i: X([p]) \to X([p-1])$ for the induced map.  The $\Delta$-complex $X$ is then determined by the sets $X([p])$ for $p \geq 0$ and the maps $d_i: X([p]) \to X([p-1])$ for $i = 0, \dots, p$.  These satisfy $d_i d_j = d_{j-1} d_i$ for $i < j$, and any sequence of sets $X([p])$ and maps $d_i$ satisfying this axiom uniquely specifies a $\Delta$-complex.

\subsection{Unordered $\Delta$-complexes}
\label{sec:symmetric}

The vertices in each simplex of a $\Delta$-complex come with a total ordering which is part of the structure, and the simplices are glued along maps which preserve this ordering.  Allowing gluing maps which do not preserve ordering leads to a more general notion of cell complex which we will call \emph{unordered $\Delta$-complex}. (This terminology is not standard, but seems reasonable to us.)  The precise definition can be given either as a topological space equipped with a set of maps from $\Delta^p$ for varying $p$, satisfying certain axioms, or, equivalently, by specifying the combinatorial gluing pattern for these simplices.  We again take the latter as the official definition.

Let $I$ be the category with the same objects as $\Delta_\inj$, but whose morphisms $[p] \to [q]$ are all
injective maps $\{0, \dots, p\} \to \{0, \dots, q\}$.

\begin{definition}
An \emph{unordered $\Delta$-complex} is a presheaf $X$ on $I$ such that the action of the symmetric group  $S_{p+1} = I([p],[p])$ on $X([p])$ is free.
\end{definition}

An unordered $\Delta$-complex $X$ has a geometric realization $|X|$ defined by the formula~(\ref{eq:realization}), just as a $\Delta$-complex does, where the equivalence relation now uses all morphisms $\theta$ in $I$. Each element $x \in X([p])$ gives rise to a map $x: \Delta^p \to |X|$, but the $(p+1)!$ elements in the $S_{p+1}$-orbit of $x$ gives rise to maps with the same image, and differ only by precomposing with an automorphism of the topological space $\Delta^p$ induced by permuting the vertices.  The intuition is that $|X|$ is again built by gluing simplices along faces, but the vertices in each simplex do not come with a preferred ordering.  The set $X([p])$ has $(p+1)!$ elements for each $p$-simplex in $|X|$, one for each total ordering of its vertices, and the $p$-simplices in $|X|$ correspond to $S_{p+1}$-orbits in $X([p])$.  If one chooses a representative for each orbit, the resulting maps $\Delta^p \to |X|$ define a CW structure on $|X|$.

Note that a $\Delta$-complex $X: \Delta_\mathrm{inj}^\mathrm{op} \to \mathrm{Sets}$ determines an
unordered $\Delta$-complex $X':I^\mathrm{op} \to \mathrm{Sets}$ with homeomorphic geometric realization, given by $X'([p]) = S_{p+1} \times X([p])$, where an element of $X'([p])$ is thought of as a simplex of $X$ together with a re-ordering of its vertices.  In more category theoretic terms, $X'$ is the \emph{left Kan extension} of $X$ along the inclusion functor $\Delta_{\inj} \to I$.  (We refer the reader to \cite{MacLane98} for details on Kan extensions and other notions from category theory.)

\begin{remark}\label{remark:simplicial-complexes}
Recall that a $\Delta$-complex $X$ is \emph{regular} if the maps $\Delta^p \to |X|$ associated to $\sigma \in X([p])$ for all $p \geq 0$ are all injective.  This definition makes sense equally well for unordered $\Delta$-complexes $X$ and is equivalent to the condition that every edge of $X$ has two distinct endpoints, i.e., for any $e \in X([1])$, we have $d_0(e) \neq d_1(e)$.

When we discuss the {dual complexes} of normal crossings divisors $D \subset X$ in a smooth variety or DM stack in \S\ref{sec:stackboundary}, we shall associate a {\em generalized $\Delta$-complex}, which we are about to define, to any such $D \subset X$.  The dual complex will be a regular unordered $\Delta$-complex exactly when $D$ has \emph{simple} normal crossings, meaning that every irreducible component of $D$ is smooth.
\end{remark}

\subsection{Generalized $\Delta$-complexes}
\label{sec:generalized}

One defect of the category of unordered $\Delta$-complexes is that many diagrams do not have colimits.  For an explicit example, consider the {\em unordered interval} defined by the representable functor $I(-,[1])$.  The geometric realization is homeomorphic to the interval $[0,1]$, but with no preferred ordering of the vertices.  There are precisely two endomorphisms of $[1] \in I$, but the resulting diagram $I(-,[1]) \double I(-,[1])$ does not have a colimit in the category of unordered $\Delta$-complexes.  One can take the colimit $X$ in the category of presheaves on $I$, but the resulting action of $S_2$ on $X([1])$ is not free.  Repairing this defect is one motivation for the following generalization of unordered $\Delta$-complexes, in the spirit of the generalized cone complexes introduced in \cite[\S2]{acp}.   (See \S\ref{sec:cones} for a discussion of the precise relationship between generalized $\Delta$-complexes and generalized cone complexes.)
 
\begin{definition}
  A \emph{generalized $\Delta$-complex} is a presheaf on $I$.
\end{definition}

\noindent In other words, a generalized $\Delta$-complex is a functor from $I^\mathrm{op}$ to sets.  Again we define the geometric realization $|X|$ by the formula~(\ref{eq:realization}), where the equivalence relation $\sim$ uses all morphisms $\theta$ in $I$.  

\begin{example}\label{ex:halfinterval}
A typical example of a generalized $\Delta$-complex which is not an unordered $\Delta$-complex is the {\em half interval} $$X = \varinjlim (I(-,[1]) \double I(-,[1]))$$ described above, given by taking $X([0])$ and $X([1])$ to be one-point sets and $X([p]) = \emptyset$ for $p \geq 2$.  The unique element in $X([1])$ gives a map $\Delta^1 \to |X|$ which is not injective; it identifies $|X|$ with the topological quotient of $\Delta^1$ by the action of $\Z/2\Z$ that reverses the orientation of the interval.  
\end{example}

For general $X$, the geometric realization $|X|$ will be built out of quotient spaces $\Delta^p/H$ for varying $p$ and subgroups $H < S_{p+1}$. Indeed, for each element $x \in X([p])$ we have the stabilizer
subgroup $H_x \subset S_{p+1}$.  Then $x$ gives a map $\Delta^p \rightarrow |X|$ which factors through the quotient $\Delta^p / H_x$, and whose restriction to $\partial \Delta^p$ maps into the union of the images of the maps $\Delta^{p-1} \rightarrow |X|$ induced by the elements $d_i x$ in $X([p-1])$.  The geometric
realization $|X|$ is therefore filtered by closed subspaces $|X|^{p}$, defined as the image of $\coprod_{i \leq p} \Delta^i \times X([i]) \to |X|$. We have $|X|^{-1} = \emptyset$ and $|X|^{p}$ is obtained from $|X|^{p-1}$ as the pushout of the diagram
\begin{equation*}
  \coprod_{x \in
  X([p])} \Delta^p/H_x \leftarrow   \coprod_{x \in
  X([p])} \partial \Delta^p/H_x \rightarrow |X|^{p-1},
\end{equation*}
where the disjoint unions are over a choice of $x \in X([p])$, one in each $S_{p+1}$-orbit and the map
$\coprod_{x \in X([p])} \partial \Delta^p/H_x \rightarrow |X|^{p-1}$ is induced by the face maps $d_i: X([p]) \to X([p-1])$.  (More canonically the disjoint unions can be written as $(\Delta^p \times X([p]))/S_{p+1}$ and similarly for $\partial \Delta^p$.)

The category of generalized $\Delta$-complexes has all small colimits, and these are computed objectwise; in other words, the set of $p$-simplices of the colimit is the colimit of the set of $p$-simplices.  In particular, any pair of parallel arrows $X \double Y$ of generalized $\Delta$-complexes has a coequalizer $Y \to Z$ which is again a generalized $\Delta$-complex.  Note that this coequalizer need not be an unordered $\Delta$-complex even when $X$ and $Y$ are, as in Example~\ref{ex:halfinterval}.  The set of $p$-simplices of the coequalizer $Z$ is the coequalizer in sets $X([p]) \double Y([p]) \to Z([p])$.

\begin{lemma}
Let $Z$ be the generalized $\Delta$-complex defined as the coequalizer of $X \double Y$.  Then the natural map from the coequalizer of $|X| \double |Y|$ in topological spaces to $|Z|$ is a homeomorphism.
\end{lemma}

\begin{proof}%[Proof sketch]
This is a special case of the more general observation that geometric realization, as a functor from generalized $\Delta$-complexes to topological spaces, preserves colimits over \emph{any} small indexing category whatsoever.  Recall that any functor which admits a right adjoint will automatically preserve all small colimits \cite[V.5]{MacLane98}.

A right adjoint to geometric realization may be defined as follows.  Let $Z$ be a topological space, and let $\mathrm{Sing}(Z)$ be the generalized $\Delta$-complex which sends $[p]$ to the set of all continuous maps $\Delta^p \to Z$.  The resulting functor $\mathrm{Sing}$ is right adjoint to geometric realization.
Indeed, given $X\colon I^\mathrm{op}\rightarrow\mathrm{Set}$ and given a topological space $Z$, a natural transformation from $X$ to $\mathrm{Sing}(Z)$ amounts to a choice of a continuous map $\Delta^p\to Z$ for every element of $X([p])$, such that the choices are compatible with the gluing data $X(\theta)$ for all $\theta\in\operatorname{Mor}(I)$.  These precisely give the data of a continuous map $|X|\rightarrow Z.$  Moreover, the association is natural with respect both to maps $X\rightarrow X'$ and to maps $Z\rightarrow Z'$.
\end{proof}

\begin{remark}\label{remark:collapse-maps}
The category of generalized $\Delta$-complexes does not have all the morphisms one might expect.  For example, the geometric realization of the representable functor $I(-,[p])$ is homeomorphic to $\Delta^p$, but the collapse map $\Delta^1 \to \Delta^0 = \{\ast\}$ is not modeled by a natural transformation $I(-,[1]) \to I(-,[0])$ since there are no such natural transformations (by the Yoneda lemma, they are in bijection with injective maps $[1] \to [0]$).

This could be remedied by including the category $I$ into the category $F$ with the same objects but with morphisms $[p] \to [q]$ being \emph{all} maps of sets, injective or not.  To a functor $X: I^\mathrm{op} \to \mathrm{Sets}$ is associated a functor $X': F^\mathrm{op} \to \mathrm{Sets}$ by \emph{left Kan extension}, i.e.,\ $X'([p]) = \colim_{[q] \to [p]} X([q])$, where the colimit is in Sets and is indexed by the category whose objects are morphisms $\theta: [q] \to [p]$ in $F$ and whose morphisms $\theta \to \theta'$ are morphisms $\phi: [q] \to [q']$ in $I$ such that $\theta' \circ \phi = \theta$.

The passage from $X: I^\mathrm{op} \to \mathrm{Sets}$ to $X': F^\mathrm{op} \to \mathrm{Sets}$ loses very little information; it is faithful and induces a bijection between isomorphisms $X \to Y$ and isomorphisms $X' \to Y'$.  Furthermore, if the geometric realization of a functor $F^\mathrm{op} \to \mathrm{Sets}$ is again defined using the formula~(\ref{eq:realization}) and all morphisms $\theta$ in $F$, then  there is a natural homeomorphism $|X| \approx |X'|$. Note, however, that this functor is not full.  There are typically more non-invertible morphisms $X' \to Y'$ than there are $X \to Y$.  For example the map $\Delta^1 \to \Delta^0$ is modeled by the unique morphism $F(-,[1]) \to F(-,[0])$.
\end{remark}

\begin{example}\label{ex:Dgn}
Let us return to the tropical moduli space $\Dgn$, which we defined in \S\ref{sec:graphs}.  To illustrate how the definitions of this section work for $\Dgn$, we will give two constructions that exhibit $\Dgn$ as the geometric realization of a generalized $\Delta$-complex.

Choose for each object $\Gmw \in J_{g,n}$ a bijection $\tau = \tau_\Gmw: E(\Gmw) \to [p]$ for the appropriate $p \geq 0$.  This chosen bijection shall be called the \emph{edge-labeling} of $\Gmw$.  Then a morphism $\phi: \Gmw \to \Gmw'$ determines an injection
\begin{equation*}
  [p'] \xrightarrow{\tau_{\Gmw'}^{-1}} E(\Gmw') \xrightarrow{\phi^{-1}} E(\Gmw)
  \xrightarrow{\tau_\Gmw} [p],
\end{equation*}
where the middle arrow is the induced bijection from the edges of $\Gmw'$ to the non-collapsed edges of $\Gmw$.  This gives a functor $F\colon J_{g,n}^\mathrm{op}\rightarrow I$ sending $\Gmw$ to the codomain $[p]$ of $\tau_\Gmw$, and hence induces a functor from $J_{g,n}^\mathrm{op}$ to generalized $\Delta$-complexes, given as $\Gmw \mapsto I(-,F(\Gmw))$, whose colimit $X=X_{g,n}$ has geometric realization homeomorphic to $\Dgn$.    
 
Basically, this is because $\Gmw \mapsto I(-,[p])$ gives an association of a $p$-simplex $|I(-,[p])|$ to $\Gmw$; here, $p+1$ is the number of edges of $\Gmw$.  Furthermore, if $\Gmw'\in J_{g,n}$ is an object with $p'+1$ edges, then the injection of label sets $[p']$ to $[p]$, determined as above by a morphism $\Gmw\to\Gmw'$, is the one that induces the appropriate gluing of the simplex $|I(-,F(\Gmw'))|$ to a face of $|I(-,F(\Gmw))|$.  In other words, it induces the gluing obtained by restricting the gluing of $\sigma(\Gmw')$ to a face of $\sigma(\Gmw)$ in the construction of $\Mgn$. 

The colimit $X\colon I^\mathrm{op} \to \mathrm{Sets}$ can be described
explicitly as follows. The elements of $X([p])$ are edge-labelings of $\Gmw\in J_{g,n}$ with the labels $[p]$, up to the equivalence relation given by the action of $\Aut(\Gmw)$ that permutes labels.  Here $\Gmw$ ranges over all objects in $J_{g,n}$ with exactly $p+1$ edges; recall that we have tacitly picked one element in each isomorphism class in $J_{g,n}$ to get a small category.

Next, for each injective map $\iota\colon [p']\to[p]$, define the following map $X(\iota) \colon X([p]) \to X([p'])$; given an element of $X([p])$ represented by $(\Gmw,\tau\colon E(\Gmw)\to [p])$, contract the edges of $\Gmw$ whose labels are not in $\iota([p'])\subset [p]$, then relabel the remaining edges with labels $[p']$ as prescribed by the map
$\iota$.  The result is a $[p']$-edge-labeling of some new object $\Gmw'$, and we set $X(\iota)(\Gmw)$ to be the element of $X([p'])$ corresponding to it.
\end{example}

\subsection{Generalized cone complexes and generalized $\Delta$-complexes}
\label{sec:cones}

We now briefly discuss the \emph{generalized cone complexes} of \cite[\S2]{acp} and their relationship to the generalized $\Delta$-complexes described here.  We will see that the category of generalized $\Delta$-complexes is equivalent to the category of {\em smooth} generalized cone complexes, by which we mean the category whose objects are generalized cone complexes built out of copies of standard orthants in $\R^n$, and whose arrows are face morphisms.  See below for a precise definition.

Recall that there is a category of \emph{cones} $\sigma$ and \emph{face morphisms} $\sigma \to \sigma'$.
A cone is a topological space $\sigma$ together with an ``integral structure'', i.e.,\ a finitely generated subgroup of the group of continuous functions $\sigma \to \R$ satisfying a certain condition. A face morphism $\sigma \to \sigma'$ is a continuous function satisfying another condition.  We shall not recall the details, since we shall not need the full category of all cones.  The special case we need is the standard orthant $(\R_{\geq 0})^{[p]} = \prod_0^p \R_{\geq 0}$ together with the abelian group $M$ generated by the $p+1$ projections $(\R_{\geq 0})^{[p]} \to \R$ onto the axes.  Any injective map $\theta: [p] \to [q]$ gives rise to an injective continuous map $(\R_{\geq 0})^{[p]} \to (\R_{\geq 0})^{[q]}$ by restricting the linear map sending the $i$th basis vector to the $\theta(i)$-th basis vector.  The face morphisms $(\R_{\geq 0})^{[p]} \to (\R_{\geq 0})^{[q]}$ are precisely the maps induced by such $\theta$.  In other words, if we write $C([p]) = (\R_{\geq 0})^{[p]}$ with this integral structure, we have defined a functor
\begin{align*}
  I & \to (\text{Cones},\text{face morphisms}) \\
  [p] & \mapsto C([p]),
\end{align*}
which is full and faithful.  Let us say that a cone is \emph{smooth} if it is isomorphic to $C([p])$ for some $p \geq 0$; then the category of smooth cones and face morphisms between them is \emph{equivalent} to $I$.

In \cite[\S2.6]{acp}, a \emph{generalized cone complex} is a topological space $X$ together with a presentation as $\colim (r \circ F)$, where $F: C \to (\text{Cones},\text{face morphisms})$ is a functor from a small category $C$, and $r$ denotes the forgetful functor from cones to topological spaces.  There is also a notion of morphism of generalized cone complexes.  We say that a generalized cone complex is \emph{smooth} if it is isomorphic to a colimit of smooth cones.

Let $X: I^\mathrm{op} \to \mathrm{Sets}$ be a generalized $\Delta$-complex.  We consider the category $C_X$, whose objects are pairs $([p],x)$ consisting of an object $[p] \in I$ and an element $x \in X([p])$, and whose morphisms $([p],x) \to ([q],y)$ are the morphisms $\theta \in I([p],[q])$ with $\theta^* y = x$.  Then there is a forgetful functor $u: C_X \to I$, and it is easy to verify that the geometric realization $|X|$ is precisely the colimit of the functor 
\begin{equation*}
  C_X \xrightarrow{u} I \xrightarrow{\Delta^\bullet} \mathrm{Top}
\end{equation*}
obtained by composing $u$ with $[p] \mapsto \Delta^p$.  If we instead take the colimit of the composition  from $C_X$ to $I$, to the category of cones, included into the category of generalized cone complexes, we get the generalized cone complex $\Sigma_X$ associated to $X$.

To define a correspondence in the other direction, we first extend the notion of ``face morphism'' between cones to morphisms between generalized cone complexes.  If $\Sigma$ is a smooth generalized cone complex and $\sigma$ is a smooth cone, let us say that a morphism $\sigma \to \Sigma$ is a \emph{face morphism} if it admits a factorization as $\sigma \to \sigma' \to \Sigma$, where the second map $\sigma' \to \Sigma$ is one of the cones in the colimit presentation of $\Sigma$ and the first map is a face morphism of cones.  If $\Sigma$ and $\Sigma'$ are generalized cone complexes, a morphism $\Sigma \to \Sigma'$ is a face morphism if the composition $\sigma \to \Sigma \to \Sigma'$ is a face morphism for all cones $\sigma \to \Sigma$ in the colimit presentation of $\Sigma$.  We may then define a functor $X_\Sigma: I^\mathrm{op} \to \mathrm{Sets}$ whose value $X_\Sigma([p])$ is the set of face morphisms $C([p]) \to \Sigma$.

These processes are inverse and give an equivalence of categories between generalized $\Delta$-complexes and the category whose objects are smooth generalized cone complexes and whose morphisms are face morphisms between such.  Geometrically, if $X:I^\mathrm{op} \to \mathrm{Sets}$ is a generalized unordered $\Delta$-complex, the geometric realization $|X|$ is the \emph{link} of the cone point in the corresponding generalized cone complex $\Sigma_X$.

\begin{remark}
The notion of \emph{morphism} between (smooth) generalized cone complexes used in \cite[\S2.6]{acp} contains many other morphisms, in addition to face morphisms.  For instance, the map $C([1]) = (\R_{\geq 0})^2 \to \R_{\geq 0}= C([0])$ given in coordinates as $(x_0,x_1) \mapsto x_0 + x_1$ is a morphism of cones and hence generalized cone complexes, but is not a face morphism.  These additional morphisms are necessary to make the construction of skeletons of toroidal varieties (and DM stacks) functorial with respect to arbitrary toroidal morphisms.
\end{remark}

\section{Cellular chains}
\label{sec:cellular-chains}

We now develop a theory of cellular chains and cochains for generalized $\Delta$-complexes, extending the usual theory of cellular chains and cochains for ordinary $\Delta$-complexes.  In general, the cellular homology of a generalized $\Delta$-complex does not agree with the singular homology of its geometric realization with integer coefficients (see Example~\ref{ex:halfintervalhomology}).  Nevertheless, we prove a comparison theorem for rational cellular and singular homology in Lemma~\ref{lem:cellularhomology}.  As explained in \S\ref{sec:residues}, cellular chains are the natural objects to consider when comparing the de Rham cohomology of a smooth open variety or DM stack with the homology of the dual complex of a normal crossings compactification using integrals of residues of logarithmic forms.  Cellular homology also offers computational advantages, since the groups of cellular chains on a generalized $\Delta$-complex typically have much smaller rank than other natural alternatives, such as the groups of simplicial chains on the barycentric subdivision.

\medskip

When $X$ is a $\Delta$-complex, the calculation of the singular homology $H_*(|X|; \Z)$ via cellular chains (sometimes called simplicial chains) is well-known; see, for instance, \cite[\S2]{Hatcher02}. In terms of the functor $X: \Delta_\inj^\mathrm{op} \to \mathrm{Sets}$, the cellular chain groups are given by
\begin{equation*}
  C_p(X) = \Z X([p]),
\end{equation*}
the free abelian group on the set $X([p])$, and the boundary map is given by the alternating sum
\begin{equation}\label{eq:boundary}
  \partial = \sum_{i=0}^p (-1)^i (d_i)_*: \Z X([p]) \to \Z X([p-1]).
\end{equation}
Cellular cochains are defined similarly and their cohomology is canonically isomorphic to $H^*(|X|;\Z)$.

\subsection{Cellular chains for generalized $\Delta$-complexes} \label{sec:generalizedcellular}

Similar formulas compute the homology and cohomology (resp. rational homology and cohomology) of $|X|$ when $X$ is an unordered $\Delta$-complex (resp. generalized $\Delta$-complex).  

Let $X$ be a generalized $\Delta$-complex, and let $\Z^\mathrm{sign}$ denote the sign representation of the symmetric group $S_{p+1}$ on the abelian group $\Z$.  

\begin{definition}
The group of cellular $p$-chains $C_p(X)$ is the group of coinvariants
\begin{equation*}
  C_p(X) = (\Z^\mathrm{sign} \otimes_\Z \Z X([p])) _{S_{p+1}},
\end{equation*}
where $\Z X([p])$ denotes the free abelian group on the set $X([p])$.  
\end{definition}

The boundary map $C_p(X) \to C_{p-1}(X)$ is defined using the diagram
\begin{equation*}
  \xymatrix{
    {\Z X([p])} \ar[rr]^-{\sum (-1)^i (d_i)_*} \ar@{->>}[d] && {\Z X([p-1])}\ar@{->>}[d]\\
    {C_p(X)} \ar[rr]^-{\partial} && C_{p-1}(X).
  }
\end{equation*}

Similarly, we define cellular chains with coefficients in an abelian group $A$ as $$C_p(X;A) = C_p(X) \otimes_\Z A,$$ and define cochains by dualizing;
\begin{equation*}
  C^p(X;A) = \Hom_\Z(C_p(X),A) = (\Hom_\Z((\Z^\mathrm{sign}\otimes \Z X([p]),A))^{S_{p+1}}.
\end{equation*}
Here, the superscript denotes the invariants for the group action.  In other words, $C^p(X;A)$ is the abelian group consisting of all set maps $\phi: X([p]) \to A$ which satisfy $\phi (\sigma x) = \mathrm{sgn}(\sigma) \phi(x)$ for all $x \in X([p])$ and all $\sigma \in S_{p+1}$.

When $X$ is an unordered $\Delta$-complex, the cellular chain complex $C_*(X)$ computes the singular homology of the geometric realization $|X|$; see Lemma~\ref{lem:cellularhomology}(1).  For arbitrary generalized $\Delta$-complexes, the cellular chain complex $C_*(X)$ is not quite so well-behaved with integer coefficients, essentially because the construction $(\Z^\mathrm{sign} \otimes_\Z -)_{S_{p+1}}$ does not have good exactness properties, as in the following example.
  
\begin{example}\label{ex:halfintervalhomology}
Let $X$ be the {\em half interval} of Example~\ref{ex:halfinterval}, with $X([0])= \{v\}$, $X([1]) = \{e\}$, and $X([p])=\emptyset$ for $p>1$.  Then
$$C_1(X;\Z) = (\Z^\mathrm{sign} \otimes \Z \!\cdot\!e)/\langle 2\otimes e\rangle \cong \Z/2\Z,$$ and $\partial\colon C_1(X;\Z) \to C_0(X;\Z)$ is zero; so $H_1(X;\Z) = \Z/2\Z$.  In contrast, $H_1(|X|;\Z) = 0$.
\end{example}

\noindent Nevertheless, we have the following comparison result for rational cellular and singular homology of arbitrary generalized $\Delta$-complexes, along with the integral comparison result for unordered $\Delta$-complexes.

\begin{lemma}  \label{lem:cellularhomology}
  The boundary homomorphism on cellular chains satisfies $\partial^2 = 0$.  Furthermore
  \begin{enumerate}
  \item If $X$ is an unordered $\Delta$-complex, there are natural isomorphisms
    \begin{align*}
      H_p(|X|;A) &\cong H_p(C_*(X;A),\partial), \mbox{ and} \\
      H^p(|X|;A) &\cong H^p(C^*(X;A),\partial),
    \end{align*}
    for any abelian group $A$.
  \item If $X$ is any generalized $\Delta$-complex, there are
    natural isomorphisms
    \begin{align*}
      H_p(|X|;\Q) &\cong H_p(C_*(X;\Q), \partial), \mbox{ and}\\
      H^p(|X|;\Q) &\cong H^p(C^*(X;\Q), \partial).
    \end{align*}
  \end{enumerate}
\end{lemma}
\begin{proof}[Proof sketch]
In both cases, the argument is similar to the classical proof for $\Delta$-complexes.  We filter the geometric realization $|X|$ by its ``skeletons'' $|X|^{p}$, where $|X|^{p}$ is the image of $\coprod_{i=0}^p X([p]) \times \Delta^i \to |X|$.  The map $X([p]) \times \Delta^p \to |X|^p$ induces a homeomorphism to the quotient space $|X|^p/|X|^{p-1}$ from the orbit space 
\begin{equation*}
    \bigg(\frac{\Delta^p \times X([p])}{\partial \Delta^p \times
      X([p])} \bigg) \Big / S_{p+1}
\end{equation*}
and hence we have a natural map
\begin{equation*}
    (\Z X([p]) \otimes H_p(\Delta^p, \partial \Delta^p))_{S_{p+1}} \to
    H_p(|X|^{p},|X|^{p-1}).
\end{equation*}
If the action of $S_{p+1}$ on $X([p])$ is free, this map is an isomorphism and $H_*(|X|^{p}, |X|^{p-1}) = 0$ for $* \neq p$, since the quotient $|X|^p/|X|^{p-1}$ becomes a wedge of $p$-spheres, one for each $S_{p+1}$-orbit in $X([p])$.  If the action is not free, this argument still applies rationally, since both sides will be a direct sum of copies of $\Q$, one for each $S_{p+1}$-orbit of elements $x \in X([p])$ whose stabilizer is contained in the alternating group.

This proves that $H_*(|X|)$ is calculated by a chain complex with groups $$H_p(|X|^p,|X|^{p-1}) \cong C_p(X),$$ and it remains to check that the boundary map is as claimed.  By linearity it suffices to verify this on generators, i.e.,\ elements of $X([p])$, and any such element is in the image of some natural transformation $I(-,[p]) \to X$.  Both the claimed formula and the actual boundary map define natural transformations $C_p(-) \to C_{p-1}(-)$ and hence suffices to prove that they agree in the case $X = I(-,[p])$. This case, in which $|X| = \Delta^p$, is proved in the same way as for $\Delta$-complexes.
\end{proof}

If $X$ is an unordered $\Delta$-complex, each element $x \in X([p])$ gives an element of $C_p(X)$, and as $x$ runs through a set of representatives for $S_{p+1}$-orbits, these elements form a basis for the free $\Z$-module $C_p(X)$.  Different choices of orbit representatives give rise to bases of $C_p(X)$ which differ only by multiplying some basis vectors by $-1$.  For generalized $\Delta$-complexes we have the following analogous result for $C_p(X)$.

\begin{lemma}  \label{lem:cellularbasis}
  Let $X$ be a generalized $\Delta$-complex.   Then $$C_p(X,\Z) \cong \Z^\alpha \oplus (\Z/2\Z)^\beta,$$ where
\begin{enumerate}
\item a basis for $\Z^\alpha$ is the set of classes $[x]\in C_p(X)$ as $x\in X([p])$ runs through a set of representatives of $S_{p+1}$-orbits whose stabilizers are contained in the alternating group, and
\item a basis for $(\Z/2\Z)^\beta$ is the set of classes $[x]\in C_p(X)$ as $x\in X([p])$ runs through a set of representatives of $S_{p+1}$-orbits whose stabilizers are not contained in the alternating group.
\end{enumerate}
In particular, a basis for $C_p(X,\Q)$ consists of elements $[x]$ as $x\in X([p])$ runs through a set of representatives of $S_{p+1}$-orbits whose stabilizers are contained in the alternating group.
\end{lemma}

\begin{proof}
This follows from the fact that the relations defining the group of coinvariants $(\Z^\mathrm{sign} \otimes_\Z \Z X([p])) _{S_{p+1}}$ as a quotient of $\Z^\mathrm{sign} \otimes_\Z \Z X([p])$ are generated by those of the form $[x] - \operatorname{sgn}(\sigma)[\sigma x]$ for $x\in X([p])$.  In particular, $[x]+[x]$ is a relation if and only if $x$ has an odd permutation in its stabilizer.
\end{proof}

\subsection{Cellular chains of coequalizers}
\label{sec:chains-of-coequalizers}

We have seen that coequalizers of generalized $\Delta$-complexes exist and commute with geometric realization.  Let us discuss the behavior of cellular chain complexes under taking coequalizer.
\begin{lemma}
Let $f,g: X \double Y$ be two parallel morphisms of generalized $\Delta$-complexes and let $Y \to Z$ be the coequalizer in generalized $\Delta$-complexes.  Then there is an exact sequence of cellular chain complexes
\begin{equation*}
    C_p(X;A) \xrightarrow{g_* - f_*} C_p(Y;A) \to C_p(Z;A) \to 0,
\end{equation*}
and similarly for cohomology.
\end{lemma}

\begin{proof}
The $S_{p+1}$-set $Z([p])$ is the coequalizer of $X([p]) \double Y([p])$.  Taking free abelian groups gives a coequalizer diagram $\Z X([p]) \double \Z Y([p]) \to \Z Z([p])$ in the category of abelian groups (because ``free abelian group'' is left adjoint to the forgetful functor), which is equivalent to the exact sequence
\begin{equation*}
    \Z X([p]) \xrightarrow{g_* - f_*} \Z Y([p]) \to \Z Z([p]) \to 0.
\end{equation*}
Tensoring this exact sequence of $\Z[S_{p+1}]$-modules with $\Z^\mathrm{sign}$, taking $S_{p+1}$-coinvariants, and tensoring with $A$ are all right exact functors, which gives the exact sequence claimed in the lemma.
\end{proof}

\begin{corollary}
Let $f,g: X \double Y$ be maps of generalized $\Delta$-complexes.  Then the rational singular homology of the coequalizer of $|X| \double |Y|$ in topological spaces is calculated by the chain complex defined by the degreewise cokernel of cellular chains $g_* - f_*: C_*(X;\Q) \to C_*(Y;\Q)$.  Similarly the rational cohomology of $|Z|$ is calculated by the degreewise kernel of cellular cochains $g^* - f^*: C^*(Y;\Q) \to C^*(X;\Q)$.
\end{corollary}

\begin{proof}
We have seen that the coequalizer of $|X| \double |Y|$ is homeomorphic to the geometric realization of the coequalizer $Y \to Z$ of $f$ and $g$ in generalized $\Delta$-complexes.  The lemma above calculates the rational homology of $Z$ as the homology of the degreewise cokernel of $g_* - f_*$ and similarly for rational cohomology.
\end{proof}

\section{Boundary complexes} \label{sec:stackboundary}

The theory of dual complexes for simple normal crossings divisors is well-known; these are regular unordered $\Delta$-complexes that have been extensively studied in algebraic geometry.  Many applications involve the fact that the homotopy types of \emph{boundary complexes}, the dual complexes of boundary divisors in simple normal crossings compactifications, are independent of the choice of compactification.  Boundary complexes were introduced and studied by Danilov in the 1970s \cite{Danilov75}, and have become an important focus of research activity in the past few years, with new connections to Berkovich spaces, singularity theory, geometric representation theory, and the minimal model program.  See, for instance, \cite{Stepanov06, Thuillier07, Stepanov08, ABW13, Kollar13, boundarycx, KapovichKollar14, KollarXu16, Simpson16, deFernexKollarXu17}.

In order to apply combinatorial topological properties of $\Dgn$ to study the moduli space of curves $\cM_{g,n}$ using the compactification by stable curves, we must account for the fact that the boundary divisor in $\ocM_{g,n}$ has normal crossings, but not simple normal crossings; its irreducible components have self-intersections and the fundamental groups of strata act nontrivially by monodromy on the analytic branches of the boundary.

In this section we explain how dual complexes of normal crossings divisors are naturally interpreted as generalized $\Delta$-complexes and, in particular, the dual complex of the boundary divisor in the stable curves compactification of $\cM_{g,n}$ is naturally identified with $\Dgn$.  We emphasize that the 
subtleties that arise come from generalizing from simple normal crossings to normal crossings divisors; passing from varieties to stacks is relatively straightforward.  

The material presented in \S\ref{sec:dual-compl-simple}-\ref{ss:homotopyinvariance} is primarily a reframing of the main results from \cite{acp} in the language of generalized $\Delta$-complexes which relies, in turn, on Thuillier's theory of skeletons for toroidal embeddings \cite{Thuillier07}.  In \S\ref{sec:simplicialinvariance} we prove an additional invariance result comparing the rational homology of the boundary complex to that of the geometric realization of the simplicial object in unordered $\Delta$-complexes associated to an \'etale cover by a smooth variety in which the preimage of the boundary divisor has simple normal crossings, and in \S\ref{sec:topweight} we state the generalization to DM stacks of the standard comparison theorem for top weight cohomology of varieties.  This generalization is proved in the appendix.

\subsection{Dual complexes of simple normal crossings divisors} \label{sec:dual-compl-simple}

We begin by recalling the notion of dual complexes of simple normal crossings divisors, using the language of regular unordered $\Delta$-complexes introduced in \S\ref{sec:cellular}.  In \S\ref{sec:dual-compl-norm}, we will explain how to interpret dual complexes of normal crossings divisors in smooth Deligne--Mumford (DM) stacks as {\em generalized} $\Delta$-complexes. Here and throughout, all of the varieties and stacks that we consider are over the complex numbers, and all stacks are separated and DM.

Let $X$ be a smooth variety.  Recall that a divisor $D \subset X$ has \emph{normal crossings} if it is formally locally isomorphic to a union of coordinate hyperplanes in affine space.  It has {\em simple normal crossings} if the irreducible components of $D$ are smooth. 
Recall that the {\em strata} of $D$ may be defined inductively as follows.  Suppose $\dim X = d$; then the $(d-1)$-dimensional strata of $D$ are the irreducible components of the regular locus of $D$; and for each $i<d-1$, the $i$-dimensional strata are the irreducible components of the regular locus of $D\setminus(D_{d-1}\cup\cdots\cup D_{i+1})$, where $D_j$ temporarily denotes the union of the $j$-dimensional strata of $D$.

If $D \subset X$ has simple normal crossings, then the dual complex $\Delta(D)$ is naturally understood as a regular unordered $\Delta$-complex whose geometric realization has one vertex for each irreducible component of $D$, one edge for each irreducible component of a pairwise intersection, and so on.  The inclusions of faces correspond to containments of strata.  Equivalently, using our characterization of unordered $\Delta$-complexes in terms of presheaves on the category $I$ given in \S\ref{sec:symmetric}, $\Delta(D)$ is the presheaf whose value on $[p]$ is the set of pairs $(Y, \phi)$, where $Y \subset D$ is a stratum of codimension $p$, i.e., codimension $p + 1$ in $X$, and $\phi$ is an ordering of the components of $D$ that contain $Y$, with maps induced by containments of strata. Dual complexes can also be defined in exactly the same way for simple normal crossings divisors in DM stacks.

\begin{remark}
In the literature, it is common to fix an ordering of the irreducible components of the simple normal crossings divisor $D$.  The corresponding ordering of the vertices induces a $\Delta$-complex structure on $\Delta(D)$.  Working with dual complexes as unordered $\Delta$-complexes is more natural, since it avoids this choice of an ordering, and is a special case of the construction of dual complexes for divisors with normal crossings (but not necessarily simple normal crossings) as generalized $\Delta$-complexes, given in \S\ref{sec:dual-compl-norm}.
\end{remark}

\subsection{Dual complexes of normal crossings divisors}
\label{sec:dual-compl-norm}

We now discuss the generalization to normal crossings divisors $D$ in a smooth DM stack $X$ which are not necessarily simple normal crossings, i.e.,\ the irreducible components of $D$ are not necessarily smooth and may have self-intersections.  This situation is more subtle, even for varieties, due to monodromy; the fundamental groups of strata may act by nontrivial permutations on the local analytic branches of the boundary divisor.  Note that, in the stack case, when the boundary strata have stabilizers, this monodromy action may be nontrivial even for zero-dimensional strata. This phenomenon appears already at the zero-dimensional strata of $\overline {\mathcal{M}}_{g}$ given by stable curves having nontrivial automorphisms, i.e.,\ the strata corresponding to (unweighted) trivalent graphs of first Betti number $g$ with nontrivial automorphisms.

Let $X$ be a smooth variety or DM stack.  Recall that a divisor $D \subset X$ has normal crossings if and only if there is an \'etale cover by a smooth variety $X_0 \rightarrow X$ in which the preimage of $D$ is a divisor with simple normal crossings.  Note that this \'etale local characterization of normal crossings divisors is the same for varieties and DM stacks.

Following \cite{acp}, the dual complex may be defined \'etale locally, in the following way. Choose a surjective \'etale map $X_0 \to X$ for which $D_0 = D \times_X X_0$ has simple normal crossings, set $X_1 = X_0 \times_X X_0$ and $D_1 = D \times_X X_1$, which has simple normal crossings, in $X_1$. The two projections $D_1 \double D_0$ give rise to two maps of unordered $\Delta$-complexes
\begin{equation}\label{eq:5}
  \Delta(D_1) \double \Delta(D_0),
\end{equation}
and we define $\Delta(D)$ to be the coequalizer of those two maps, in the category of generalized $\Delta$-complexes.  It is shown in \cite{acp} (in the language of generalized cone complexes) that, up to isomorphism, the resulting generalized $\Delta$-complex does not depend on the choice of $X_0 \to X$.  This recipe makes sense also in the more general case where $X$ is a smooth DM stack and $D \subset X$ is a normal crossings divisor, using a sufficiently fine \'etale atlas $X_0 \to X$, and $\Delta(D)$ is independent of the choice of $X_0 \to X$ in this generality as well.

We now give an equivalent, and more direct, description of $\Delta(D)$ as a functor $I^\mathrm{op} \to  \mathrm{Sets}$.  Let $\EE \to D$ be the \emph{normalization} of $D$.  Note that $\EE$ is smooth because $D$ has normal crossings.  In the special case where $D$ has simple normal crossings, $\EE$ is simply the disjoint union of the irreducible components of $D$.  Let us write $\EEE{p} = \EE \times_X \dots \times_X \EE$ for the $(p+1)$-fold iterated fiber product, and $\FFF{p} \subset \EEE{p}$ for the open subset consisting of $(p+1)$-tuples of distinct points whose image in $D$ lies in a stratum of codimension $p$. We then  define $\Delta(D)$ to be the functor $I^\mathrm{op} \to\mathrm{Sets}$ which sends $[p]$ to the set of irreducible components of $\FFF{p}$.

It is not difficult to show that this generalized cone complex $\Delta(D)$ agrees with the unordered $\Delta$-complex defined in \S\ref{sec:dual-compl-simple} when $D$ has simple normal crossings.  Indeed, in this special case, if $D$ has smooth components $D_0, \ldots, D_r$ then $\FFF{p}$ may be indentified with a dense open subset of the variety
  \begin{equation*}
    \coprod_{i: [p] \to [r] } D_{i(0)} \cap \dots \cap D_{i(p)},
  \end{equation*}
where the disjoint union is over injective maps $[p] \to [r]$.  Each $D_i = D_{i(0)} \cap \dots \cap D_{i(p)} \subset X$ is smooth and has codimension $p+1$ in $X$, but need not be connected.  Each component of $D_i$ is the closure of precisely one stratum $Y = Y_i \subset D$ of codimension $p+1$, and the function $i$ gives an ordering on the $p+1$ components of $D$ which contain $Y$.  Hence we have produced a bijection to the set of $p$-simplices of the unordered $\Delta$-complex described in \S\ref{sec:dual-compl-simple}, and it is easy to see that this bijection is natural with respect to maps $[p'] \to [p]$ in $I$.

\begin{lemma}\label{lemma:etale-descent}
The association $D \mapsto \Delta(D)$ satisfies \'etale descent in the sense that if $X_0 \to X$ is an \'etale cover, $X_1 =X_0 \times_X X_0$, $D_0 = D \times_X X_0$, $D_1 = D\times_X X_1$, then
\begin{equation*}
       \Delta(D_1) \double \Delta(D_0) \to\Delta(D)
\end{equation*}
is a coequalizer diagram. 
\end{lemma}

\begin{proof}
Let us write $\FFFF{0}{p} \subset \EEEE{0}{p}$ and $\FFFF{1}{p} \subset \EEEE{1}{p}$ for the corresponding construction of $\EEE{p}$ applied to $D_0 \subset X_0$ and $D_1 \subset X_1$. Since normalization is \'etale local, we then have three pullback squares
  \begin{equation*}
    \xymatrix{
      \EEEE{1}{p} \ar@/^/[r]\ar@/_/[r] \ar[d] &
      \EEEE{0}{p}\ar[d] \ar[r] &
      \EEE{p} \ar[d] \\
      X_1 \ar@/^/[r]\ar@/_/[r] &
      X_0 \ar[r] &
      X.
    }
  \end{equation*}
Since $\FFF{p} \subset \EEE{p}$ is defined by a property verified in fibers over $X$, there is a similar diagram
  \begin{equation*}
    \xymatrix{
      \FFFF{1}{p} \ar@/^/[r]\ar@/_/[r] \ar[d] &
      \FFFF{0}{p}\ar[d] \ar[r] &
      \FFF{p} \ar[d] \\
      X_1 \ar@/^/[r]\ar@/_/[r] &
      X_0 \ar[r] &
      X
    }
  \end{equation*}
in which all three squares are again pullback, all entries are smooth varieties, and the vertical maps are immersions of codimension $p+1$.  Since $X_1 = X_0 \times_X X_0$ is also a fiber product, the diagram may be reinterpreted as a 3-dimensional cubical diagram (the bottom square is the fiber product defining $X_1$), and where all sides except the top are defined to be pullback.  It follows that the top is also pullback, i.e.,\ $\FFFF{1}{p} = \FFFF{0}{p} \times_{\FFF{p}} \FFFF{0}{p}$.  This pullback diagram may be restricted to the generic points of the components of $\FFF{p}$.  Since the set of generic points of components of $\FFF{p}$ is precisely $\Delta(D)([p])$ and similarly for $\FFFF{0}{p}$ and $\FFFF{1}{p}$, we see that $\Delta(D_1)([p])$ surjects onto the fiber product of sets $(\Delta(D_0)([p])) \times_{\Delta(D)([p])} (\Delta(D_0)([p]))$. It follows that the coequalizer of $\Delta(D_1)([p]) \double \Delta(D_0)([p])$ injects into $\Delta(D)([p])$.  It is also easy to see that $\Delta(D_0)([p]) \to \Delta(D)([p])$ is surjective, using that $X_0 \to X$ is a surjective \'etale map.  This finishes the proof.
\end{proof}

As an immediate consequence, we see that the equivalence of categories between generalized $\Delta$-complexes and smooth generalized cone complexes takes $\Delta(D)$ to the skeleton $\Sigma(X)$ associated to the toroidal structure induced by the normal crossings divisor $D$, as defined in \cite{acp}.

\begin{corollary}
Let $X$ be a smooth variety or DM stack with the toroidal structure induced by a normal crossings divisor $D \subset X$.  Then the dual complex $\Delta(D)$ is the generalized $\Delta$-complex associated to the smooth generalized cone complex $\Sigma(X)$.
\end{corollary}

Most important for our purposes is the special case where $X = \ocM_{g,n}$ is the Deligne--Mumford stable curves compactification of $\cM_{g,n}$ and $D = \ocM_{g,n} \smallsetminus \cM_{g,n}$ is the boundary divisor.

\begin{corollary}\label{cor_dual_complex}
The dual complex of the boundary divisor in the moduli space of stable curves with marked points $\Delta(\ocM_{g,n} \smallsetminus \cM_{g,n})$ is $\Dgn$.
\end{corollary}

We note that the generalized $\Delta$-complex $\Delta(D)$ associated to a normal crossings divisor $D \subset X$ may also be interpreted in the following more transcendental and geometric way.  For a point $x \in D$ the \emph{local branches} of $D$ at $x$ are the germs of locally closed smooth analytic codimension 1 submanifolds $V \subset X$ containing $x$ and contained in $D$.  The codimension $p$ stratum of $D$ is a smooth locally closed subvariety $D^{(p)} \subset X$ which is locally cut out by $p+1$ equations, and admits a $(p+1)!$ sheeted cover by $\FFF{p}$.  (Note that the points of $\FFF{p}$ may be identified with pairs of a point $x \in D^{(p)}$ together with a total order of the set of local branches at $x$.)  

Identifying irreducible components with path components in the analytic topology, we conclude that $\Delta(D)([p])$ is in bijection with the set of equivalence classes $[(x,\sigma)]$ of pairs where $x \in D$ is a point in the codimension $p$ stratum and $\sigma$ is a total ordering of the set of local branches at $x$.  Two such pairs $(x,\sigma)$ and $(x',\sigma')$ are identified if there exists a continuous path $\gamma: [0,1] \to D^p$ from $x$ to $x'$, together with choices of total orderings of the local branches of $D$ at $\gamma(t)$ for all $t$, starting at $\sigma$ and ending at $\sigma'$ and depending continuously on $t$.
 
For $(x,\sigma) \in \FFF{p}$ the covering space $\FFF{p} \to D^{(p)}$ gives rise to a ``monodromy'' homomorphism $\pi_1^{et}(D^{(p)},x) \to S_{p+1}$ whose image may be identified with the stabilizer of $[(x,\sigma)] \in \Delta(D)([p])$.  Hence $\Delta(D)$ is a (non-generalized) unordered $\Delta$-complex if and only if all monodromy homomorphisms are trivial, i.e.,\ $\FFF{p} \to D^{(p)}$ is a trivial covering space for all $p$.

\begin{example}\label{ex:whitney} Consider the {\em Whitney umbrella} $D = \{x^2y=z^2\}$ in $X=\mathbb{A}^3\setminus \{y=0\},$ as in \cite[Example 6.1.7]{acp}. Then the dual complex $\Delta(D)$ is the half segment of Example~\ref{ex:halfinterval}. We will explain this calculation three times in order to demonstrate the equivalent constructions of the boundary complex.  

Let $X_0\cong \mathbb{A}^2\times \mathbb{G}_m \rightarrow X$ be the degree 2 \'etale cover given by a base change $y = u^2$.  Then $D_0 = \{x^2u^2 - z^2 = 0\}$ is simple normal crossings, and $D_1 = D_0\times_X D_0 \cong D_0\times \Z/2\Z$, since $D_0$ is degree 2 over $D$.  Explicitly, one component of $D_1$ parametrizes pairs $(p,p)$ of points in $D_0$, and the other parametrizes pairs $(p,q)$ with $p\ne q$ lying over the same point of $D$. So $\Delta(D_0)$ is a segment and $\Delta(D_1)$ is two segments, and the two maps $\Delta(D_1)\double \Delta(D_0)$ differ by one flip, making the coequalizer a half segment.

Second, we have the normalization map $E=\mathbb{A}^2_{x,u}-\{u=0\}$ sending $(x,u)$ to $(x,u^2, xu)$.  Then $F([0])=E$, while $F([1])$ is isomorphic to the 1-dimensional stratum $Y=\{x=z=0\}\cong \mathbb{G}_m$ in $X$; it has closed points $((0,u),(0,-u)) \in E\times_X E.$  So $F([0])$ and $F([1])$ each have a single  irreducible component, which completely determines $\Delta(D)$.   

Rephrased complex-analytically: the points $(0,u)$ and $(0,-u)$ in $E$ correspond to the two analytic branches along $Y$ at the point $(0,y,0)$, where $y=u^2$.  The equations of the branches are $z=xu$ and $z=-xu$.  So taking $y$ around a loop around the punctured complex plane precisely interchanges the branches.  Therefore there is only one equivalence class of pairs $(x,\sigma)$ and hence, again, only one element in $\Delta(D)([1])$.  We conclude again that $\Delta(D)$ is a half segment.
\end{example}

\subsection{Homotopy invariance}  \label{ss:homotopyinvariance}

We now focus our attention on the important special case where the normal crossings divisor is the boundary of the compactification of a smooth variety or stack $X$.  In this case, the homotopy type of the geometric realization of the boundary complex is an invariant of $X$ itself, as we now discuss.

Let $X$ be a smooth variety and $\overline X$ a compactification whose boundary $\partial \overline X = \overline X \smallsetminus X $ is a divisor with simple normal crossings.  The {\em boundary complex} of $X$  is the dual complex $\Delta(\partial \overline X)$ of the boundary divisor, and the homotopy type of its geometric realization depends only on $X$ itself and not the choice of compactification.  There are three essentially different proofs of the independence of homotopy type of the boundary complex, one using resolution of singularities and the nerve of the category of irreducible varieties mapping into the boundary \cite{Danilov75}, one using Berkovich spaces and Thuillier's deformation retraction onto skeletons of toroidal embeddings \cite{Thuillier07}, and one using weak factorization of birational maps \cite{Stepanov06, boundarycx}.  This last approach, applying the toroidal weak factorization theorem to the birational map between two simple normal crossing compactifications, gives a slightly stronger result, that the \emph{simple homotopy type} of the boundary complex is independent of the choice of compactification.  

We define the boundary complex of a smooth DM stack with a normal crossings compactification in exactly the same way; if $\cX$ is a smooth stack and $\ocX$ is a smooth and proper stack containing $\cX$ as a dense open substack, in which the complement $\cD = \ocX \smallsetminus \cX$ is a divisor with normal crossings, then the boundary complex of $\cX$ is $\Delta(\cD)$, the dual complex of the boundary divisor.  In this case also, just as for varieties, the homotopy type of the geometric realization of the boundary complex is independent of the choice of compactification.   Note, however, that we have no suitable analogue of weak factorization for birational maps of DM stacks.  It is an open problem whether the simple homotopy type of the boundary complex of a smooth DM stack is independent of the choice of compactification.  Nevertheless, the invariance of the ordinary homotopy type of boundary complexes for DM stacks may be deduced from either of the other two methods, with no new difficulties.  In particular, Thuillier's construction is generalized to DM stacks without any essential changes in \cite{acp}.

Because the homotopy type of the geometric realization of the boundary complex does not depend on the choice of compactification, its topological invariants gain interesting interpretations as algebraic invariants of $\cX$.  For instance, the following section explains that the rational reduced homology of $\Delta(\cD)$ is naturally identified with the top weight cohomology of $\cX$.  

Note that the reduced rational homology of the generalized $\Delta$-complex $\Delta(\cD)$ can be computed using cellular chains, by Lemma~\ref{lem:cellularhomology}.  Furthermore, if we choose, for each stratum of codimension $p+1$ on which the monodromy acts by even permutations on the local branches of $\cD$, a point $x$ in the stratum and an ordering of analytic branches $\sigma$ at $x$, then the corresponding classes $[(x,\sigma)]$ form a basis for the group of rational cellular $p$-chains $C_p(\Delta(D); \Q)$, by Lemma~\ref{lem:cellularbasis}.

\subsection{Top weight cohomology}   \label{sec:topweight}

Let $\cX$ be a smooth DM stack of dimension $d$ over $\C$.  The rational singular cohomology of $\cX$, like the rational cohomology of a smooth variety, carries a canonical mixed Hodge structure, in which the weights on $H^k$ are between $k$ and $\min\{2k,2d\}$.  Since the graded pieces $\Gr_j^W H^*(\cX; \Q)$ vanish for $j > 2d$, we refer to $\Gr_{2d}^W H^*(\cX; \Q)$ as the \emph{top weight cohomology} of $\cX$.  The standard identification of the top weight cohomology of a smooth variety with the reduced homology of its boundary complex carries through essentially without change for DM stacks.  It is stated in the next proposition. For further details and references, see the appendix.  

\begin{proposition} \label{prop:topweight}
Let $\cX$ be a smooth and separated DM stack of dimension $d$ with a normal crossing compactification $\overline \cX$ and let $\cD = \overline \cX \smallsetminus \cX$.  Then there is a natural isomorphism
\[
\Gr_{2d}^W H^{2d-i}(\cX; \Q) \  \cong  \ \widetilde H_{i-1}(\Delta (\cD); \Q).
\]
\end{proposition}

\noindent We sketch two proofs of this proposition in \S\ref{app:residues} and \S\ref{app:topweight}.  The proof in \S\ref{app:residues} gives more refined information about the existence of logarithmic forms with prescribed residues, as discussed in \S\ref{sec:log}, and is most naturally expressed in the language of cellular chains and cochains on generalized $\Delta$-complexes as developed in \S\ref{sec:cellular}.

\subsection{Additional invariance properties in rational homology} \label{sec:simplicialinvariance}

We briefly pause to discuss a different construction of a topological space whose rational homology agrees with that of the dual complex $\Delta(\cD)$ for a normal crossings divisor in a smooth DM stack; this construction is natural from the point of view of simplicial schemes, but the homotopy type depends on the choice of an \'etale cover.

Let $\cD \subset \cX$ be a normal crossings divisor in a smooth DM stack, and let $V \to \cX$ be an \'etale cover in which the preimage of $\cD$ has simple normal crossings.  We may then construct a simplicial scheme $V_\bullet$ with $V_0 = V$ and $V_p = V \times_\cX \dots \times_\cX V$ the $(p+1)$-fold iterated fiber product.   Then $D_p = \cD \times_\cX V_p$ has simple normal crossings in $V_p$, and we have a diagram of generalized $\Delta$-complexes
\begin{equation*}
  \xymatrix{
    \dots \ar@<1ex>[r]\ar[r]\ar@<.5ex>[r] & \Delta(D_1) \ar@<.5ex>[l] \ar@<1ex>[l]
 \ar[r] \ar@<.5ex>[r] &
    \Delta(D_0) \ar@<.5ex>[l]\ar[r] & \Delta(\cD),
  }
\end{equation*}
in which we have defined $\Delta(D_0) \to \Delta(D)$ to be the coequalizer of $\Delta(D_1) \double \Delta(D_0)$.  Passing to geometric realizations gives us a simplicial object $[p] \mapsto |\Delta(D_p)|$ in the category of topological spaces.  Taking geometric realization again gives us a (typically infinite dimensional) topological space which we shall temporarily denote $\|\Delta(D_\bullet)\|$.  It comes equipped with a map
\[
\|\Delta(D_\bullet)\| \rightarrow |\Delta(\cD)|.
\]

It is clear that the homeomorphism type of $\|\Delta(D_\bullet)\|$, unlike that of $|\Delta(\cD)|$, will depend on the \'etale cover $V \to X$, and not just the divisor $\cD$.  In fact, even the homotopy type of $\|\Delta(D_\bullet)\|$ will depend on the choice of \'etale cover, as the following example extracted from \cite[Example~6.1.10]{acp} shows.  

\begin{example}
Let $X = \A^1 \times \G_m$ and let $D$ be the divisor $\{0\} \times \G_m$.  Note that $D$ is a simple normal crossings divisor and $|\Delta(D)|$ is a single point.  If we take the identity map as an \'etale cover, then $|\Delta(D_p)|$ will be a point for all $p$, and the geometric realization of the resulting simplicial space will be a point.  

On the other hand, if we take as an \'etale cover the degree 2 map $\A^1 \times \G_m \to X$ given by $(z, t) \mapsto (z, t^2)$ then $|\Delta(D_p)|$ is discrete and naturally identified with $\mu_2^p$, where $\mu_2 = \{\pm 1\} \subset \G_m$.  In fact it can be checked that the simplicial object $[p] \mapsto |\Delta(D_p)|$ is isomorphic to the nerve of the group $\mu_2$ and hence $\|\Delta(D_\bullet)\|$ is the classifying space $B\mu_2 \simeq \R P^\infty$.
\end{example}

Nevertheless, the rational homology of $\|\Delta(D_\bullet)\|$ is independent of  the choice of \'etale cover, and agrees with that of $|\Delta(\cD)|$.
 
\begin{proposition}\label{p:simplicial}
 The natural map $\|\Delta(D_\bullet)\| \rightarrow |\Delta(\cD)|$ induces an isomorphism in rational homology.
\end{proposition}

\noindent See \S\ref{sec:simplicialproof} for the proof of Proposition~\ref{p:simplicial}.

\section{Residue integrals and torus classes} \label{sec:log}

In \S\ref{sec:topweight}, we discussed the natural isomorphism between the top weight cohomology of an open variety or stack and the reduced homology of its boundary complex.  Using the identification of $\Dgn$ with the boundary complex of $\cM_{g,n}$, this is already enough to prove cohomological results such as Theorem~\ref{thm:cohomology}, but falls short of what is needed to prove more refined statements, such as Theorem~\ref{thm:dualbasis}.  (Both of these theorems are proved in \S\ref{sec:proofs}.)

In this section, we discuss the background needed to relate closed logarithmic forms and explicit geometric cycles that we call \emph{torus classes} in an open DM stack with a normal crossings compactification to cellular chains and cochains in its boundary complex, as required for Theorems~\ref{thm:dualbasis} and \ref{thm:abelianrelations}.  

Our approach follows Deligne's description of the mixed Hodge structure of a smooth variety with a simple normal crossings compactification in terms of logarithmic forms and residues \cite[\S3]{Deligne71} which works without any essential changes for normal crossings compactifications of smooth DM stacks.  The interpretation of residue integrals in terms of cellular chains on the boundary complex is essentially standard in the case of simple normal crossings (e.g., see \cite{Hacking08}), and our definition of cellular chains in the normal crossings case, given in \S\ref{sec:generalizedcellular}, is specifically chosen to be compatible with the theory of residues; no essentially new ideas are required.  Nevertheless, lacking a suitable reference in the generality that we require for applications to moduli spaces of curves (normal crossings compactifications of smooth DM stacks), we provide further details and proofs in the appendix.

\bigskip

\subsection{Residue integrals}  \label{sec:residues}

Let $\cX$ be a smooth DM stack with compactification $\ocX$ whose boundary $\cD = \ocX \setminus \cX$ is a divisor with normal crossings, but not necessarily simple normal crossings.  We refer to the appendix for a brief discussion of the logarithmic de Rham complex $\cA^\bullet_{\ocX}(\log \cD)$, and its weight filtration. As explained there, we have a \emph{residue integral} homomorphism
\begin{equation}
  \label{eq:8}
  \int_{\widetilde \cD^{(k)}} \mathrm{res}_{\widetilde \cD^{(k)}}: \Gamma(\ocX, W_k \cA^{2d -
    k}_{\ocX}(\log \cD)) \to C_{k-1}(\Delta(\cD);\C),
\end{equation}
obtained, roughly speaking, by composing the \emph{Poincar\'e residue} with integration of differential forms along $\widetilde \cD^{(k)}$, the normalization of the closure of the codimension $k$ stratum $\cD^{(k)}$.

By Lemma~\ref{lem:cellularbasis}, the group of cellular chains $C_{k-1}(\Delta(\cD); \C)$ is isomorphic to $\C^\alpha$, where $\alpha$ is the number of components $\cY$ in $\widetilde \cD^{(k)}$ such that the monodromy representation $\pi_1(\cY) \rightarrow S_k$, given by permutation of the local branches of $\cD$, has image contained in the alternating subgroup. More precisely, if we choose a point $y$ in each such $\cY$ and an ordering $\sigma$ of the local branches of $\cD$ at $y$, then a set of representatives for the $S_k$-orbits of these $[(y,\sigma)] \in \Delta(\cD)([k])$ forms a basis for the cellular chain group $C_{k-1}(\Delta(\cD); \C)$.

The residue integral homomorphism may be described with respect to this basis, as follows.  The Poincar\'e residue $\res_\cY$ of a form $\omega \in \Gamma(\ocX, W_k \cA^{2d - k}_{\ocX}(\log \cD))$ is a smooth differential $(2d-2k)$-form on $\cY$ with coefficients in the rank 1 local system $L_\cY$ given by the sign of the monodromy representation. When this monodromy is alternating, the choice of an ordering of the local branches of $\cD$ at a point induces a trivialization of $L_\cY$, allowing us to integrate the Poincar\'e residue along $\cY$ and get a complex number.  Changing the ordering of the local branches multiplies this number by the sign of the permutation, and it follows that the cellular chain $(\int_\cY \res_\cY \omega) \cdot [(y,\sigma)] \in C_{k-1}(\Delta(\cD), \C)$ depends only on $\omega$ and $\cY$, and not the choice of $[(y,\sigma)]$.  The global residue integral $\int_{\widetilde \cD^{(k)}} \res_{\widetilde \cD^{(k)}}$ is then defined by taking the sum over all irreducible components for which the monodromy is alternating.

\begin{example}\label{ex:whitneyintegration} 
Returning to Example~\ref{ex:whitney} of the {Whitney umbrella} $D = \{x^2y-z^2 = 0\}$ in $\overline X=\mathbb{A}^3\setminus \{y=0\},$ let $Y = \{x=z=0\}$ and let $y\in Y$.  Let $V\rightarrow \overline X$ be the degree 2 \'etale cover given by a base change $y = u^2$.  Then $\sigma$ is a segment and monodromy acts on $\sigma$ by a flip.  By Lemma~\ref{lem:cellularbasis} it follows that $[\sigma] = 0$ in $C_1(\Delta(\cD),\C)$, and so the integral of any twisted differential form on $Y$ vanishes.
\end{example}

The main properties of the residue integral homomorphism are summarized in the following proposition, which we prove in the appendix.

\begin{proposition} \label{prop:cyclesandboundaries}  
Let $\cX \subset \ocX$ be as above, i.e., $\ocX$ is a proper smooth DM stack of dimension $2d$, and $\cX = \ocX \setminus \cD$ is the complement of the normal crossings divisor.

\begin{enumerate}[(i)]
  
\item For every $F \in \widetilde C_{k-1}(\Delta(\cD);\C)$ there exists a form $\omega \in \Gamma(\ocX, W_k \cA^{2d - k}(\log \cD))$ such that $d\omega$ has weight $k-1$ and $\int_{\cD^{(k)}} \mathrm{res}(\omega) = F$.

\item If $F \in \widetilde C_{k-1}(\Delta(\cD);\C)$ is a cycle, then there exists an $\omega \in \Gamma(\ocX, W_k \cA^{2d - k}(\log \cD))$ such that $\int_{\cD^{(k)}} \mathrm{res}(\omega) = F$ and $d\omega = 0$.

\item  $F \in \widetilde C_{k-1}(\Delta(\cD);\C)$ is a boundary, then there exists an $\omega \in \Gamma(\ocX, W_k \cA^{2d - k}(\log \cD))$ such that $\int_{\cD^{(k)}} \mathrm{res}(\omega) = F$ and $\omega = d\tau$ for some $\tau \in \Gamma(\ocX, W_{k+1} \cA^{2d - k-1}(\log \cD))$.
\end{enumerate}
\end{proposition}

\subsection{Torus classes} \label{sec:torus}

The residue integral of a logarithmic $(2d-k)$-form of weight $k$, as discussed in the previous section, is a cellular $(k-1)$-chain on $\Delta(\cD)$.  We now focus on the special case where $k = d$ and present a geometric interpretation for the resulting pairing between residue integrals of $d$-forms of weight $d$ and cellular $(d-1)$-cochains on $\Delta(\cD)$, given by integrating along real tori mapping into the open substack $\cX \subset \ocX$.

For each zero stratum $x$ in $\cD$ for which the monodromy group $\pi_1(y)$ (i.e., the stabilizer of the point $x$ in the stack $\ocX$) acts by alternating permutation on the local branches of $\cD$, choose an ordering $\sigma$ of these branches.  Then the cellular chains $[(x, \sigma)]$ give a basis for $C_{d-1} \Delta(\cD, \Q)$.  We consider the dual basis for the group of cellular cochains $C^{d-1}(\cD, \Q)$, and write $\delta_x$ for the basis cochain dual to $[(x,\sigma)]$.  (Note that $\delta_x$ depends on the choice of ordering $\sigma$; permuting the local branches at $x$ will multiply $\delta_x$ by the sign of the permutation.)

For each basis cochain $\delta_x$, choose an \'etale neighborhood $f: V \rightarrow \ocX$, such that $V$ is a variety, $x$ has a unique preimage $y$, and the preimage of $\cD$ is a divisor with simple normal crossings.  Let $D_1, \ldots, D_d$ be the irreducible components of $f^{-1}(D)$ that contain $y$.  Then we can choose local holomorphic coordinates $z_1, \ldots, z_d$ so that $D_i$ is the vanishing locus of $z_i$, for $1 \leq i \leq d$, with the ordering induced by $\sigma_x$.

For sufficiently small positive $\epsilon$, the real torus
\[ 
T_\epsilon = \Big\{ (z_1, \ldots, z_d) \ | \ |z_i| = \epsilon \mbox{ for } 1 \leq i \leq d \Big\}
\]
is contained in $f^{-1}(\cX)$.  The ordering of the coordinates induces an orientation on $T_\epsilon$, and the push forward of the fundamental class $f_*[T_\epsilon] \in H_d (\cX, \Q)$ depends only on the cellular cochain $\delta_x \in C^{d-1}(\Delta(\cD), \Q)$, i.e.,~it is independent of the choice of neighborhood and the choice of $\epsilon$.  It is also independent of the choice of ordering $\sigma$, in the sense that permuting the branches induces multiplication by the sign of the permutation on both $\delta_x$ and $f_*[T_\epsilon]$.  Extending linearly, we obtain a homomorphism $C^{d-1}(\Delta(\cD); \Q) \rightarrow H_d(\cX, \Q)$.  We refer to the classes in the image as \emph{torus classes} and write $\mathfrak{t}_\delta$ for the torus class that is the image of $\delta \in C^{d-1}(\Delta(\cD); \Q)$.

By Cauchy's formula and the definition of Poincar\'e residues, we have
\[ 
\textstyle{\big \langle \delta_x, \int_x \res_x \omega \big \rangle = \frac{1}{(2 \pi \sqrt{-1})^d} \cdot \int_{\mathfrak t_{\delta_x}} \omega}
\] 
for any logarithmic $d$-form $\omega$ in $\Gamma(\ocX, \cA^d_{\ocX} \log(D))$.  By linearity, this gives 
\[ 
\textstyle{\big \langle \delta, \int_{\widetilde \cD^{(d)}} \res_{\widetilde \cD^{(d)}} \omega \big \rangle = \frac{1}{(2 \pi \sqrt{-1})^d} \cdot \int_{[\mathfrak t_\delta]} \omega}
\] 
for any logarithmic $d$-form $\omega$ and any cellular cochain $\delta \in C^{d-1}(\Delta(\cD); \Q)$.

Note that, since $\res_{\widetilde \cD^{(d)}}$ vanishes on $W_{d-1} \Gamma(\ocX, \cA^d_{\ocX} \log(D))$, these torus classes pair naturally with $\Gr_{2d}^W H^d(\cX, \C)$.  We have the following immediate consequence of Propositions~\ref{prop:topweight} and \ref{prop:cyclesandboundaries}.

\begin{corollary} \label{cor:relations} The torus classes associated to cellular $(d-1)$-cochains on $\Delta(\cD)$ generate $\Gr_{2d}^W H^d(\cX, \Q)^\vee$, and the torus class $[\mathfrak t_\delta]$ vanishes in $H_d(\cX, \Q)$ if and only if $\delta$ is a coboundary.
\end{corollary}

\noindent This corollary gives a natural identification of the relations between torus cycles in middle degree on $\cM_{g,n}$ in terms of coboundaries on $\Delta_{g,n}$.  In \S\ref{sec:abcycle}, we reinterpret this statement in terms of abelian cycles for the mapping class group.

\medskip

\begin{remark}  \label{rem:torus}
We can also define torus classes below the middle degree, corresponding to cellular cochains of degree less than $(d-1)$, as follows.  On each stratum $\cY$ of codimension $k$ for which $\pi_1^{et}(\cY)$ acts by alternating permutations on the local branches of $\cD$, choose a point $y$ and an ordering $\sigma$ of the local branches at $y$.  The classes $[(y, \sigma)]$ form a basis for the group of cellular chains $C_{k-1}(\Delta(\cD), \Q)$.  We consider the dual basis for cellular cochains, and write $$\delta_\cY \in C^{k-1}(\Delta(\cD, \Q))$$ for the basis element dual to $[(y, \sigma)]$.  (Note, once again, that this cochain $\delta_\cY$ depends on the choice of the ordering $\sigma$, up to a factor of the sign of the reordering.)

We have a natural degree $k$ class $[\mathfrak t_{\delta_\cY}]$, depending only on the cellular cochain $\delta_\cY$, constructed as follows.  Choose an \'etale neighborhood $f: V \rightarrow \cX$, finite over an open subset that meets $\cY$, such that $V$ is a variety and $f^{-1}(\cD)$ is a divisor with simple normal crossings, and fix a point $z \in V$ that maps to $y$.  Choose local holomorphic coordinates $z_1, \ldots, z_d$ in a neighborhood of $z$ such that $f^{-1}(\cD)$ is the vanishing locus of $z_1 \cdots z_k$, and the ordering of $z_1, \ldots, z_k$ is given by $\sigma$.  Then the real $k$-torus
\[
T_{\epsilon} = \Big\{ (z_1, \ldots, z_d) \ \Big | \ |z_i| = \left\{ \begin{array}{ll} \epsilon &\mbox{ for } 1 \leq i \leq k, \\
			    0 & \mbox{ for } i > k, \end{array} \right. \Big\}
\] 
is contained in $f^{-1}(\cX)$ for $\epsilon$ sufficiently small, the ordering of the coordinates induces an orientation on $T_\epsilon$, and the homology class
\[
[\mathfrak t_{\delta_\cY}] = f_*[T_\epsilon]
\] 
depends only on the $(k-1)$-cochain $\delta_\cY$, in the sense that reordering the branches at $y$ multiplies both $\delta_\cY$ and $f_*[T_\epsilon]$ by the sign of the permutation.  Extending linearly gives a canonical homomorphism $C^{k-1}(\Delta(\cD), \Q) \rightarrow H_k(\cX, \Q)$.  We refer to the classes in the image as \emph{torus classes}, and these classes pair naturally with $\Gr_{2k}^W H^k(\cX, \Q)$.
\end{remark}

\section{Abelian cycles for the pure mapping class group} \label{sec:abcycle}

We now discuss the natural identification of torus classes in $\cM_{g,n}$ associated to cellular cochains on $\Dgn$ with abelian cycles for the pure mapping class group.  Using this identification, we describe all relations among abelian cycles in middle cohomological degree in terms of cellular coboundaries in the tropical moduli space $\Dgn$.

\bigskip

Let $\cY \subset \ov \cM_{g,n}$ be a boundary stratum of codimension $k$, and let $\gamma_1, \ldots, \gamma_k$ be disjoint simple closed curves on an oriented surface $S$ of genus $g$ with $n$ marked punctures such that the space obtained by collapsing the curves $\gamma_1, \ldots, \gamma_k$ to nodes and filling the punctures with marked points has the topological type of the semistable algebraic curves parametrized by $\cY$.  Let $\Mod(S)$ be the pure mapping class group of $S$, i.e.,\ the subgroup of the mapping class group respecting the marked punctures, and let $T_i \in \Mod(S)$ be the right-handed Dehn twist along $\gamma_i$.  Since the curves $\gamma_1, \ldots, \gamma_k$ are disjoint, the Dehn twists $T_1, \ldots, T_k$ commute, inducing a group homomorphism $\Z^k \rightarrow \Mod(S)$.  The push forward of the fundamental class $[\Z^k]$ is the \emph{abelian cycle} denoted $\{ T_1, \ldots, T_k \} \in H_k(\Mod(S); \Q)$.  See \cite{BrendleFarb07, ChurchFarb12} for further details and background on abelian cycles.

Suppose $\pi_1^{et}(\cY,y)$ acts by even permutations on the branches of $\cD$.  Note that the ordering of the curves $\gamma_1, \ldots, \gamma_k$ determines an ordering of the local branches of $\cD$ at $y$.  As in Remark~\ref{rem:torus}, these data determine a cellular cochain $\delta_\cY \in C^{k-1}(\Delta(\mathcal{D}))$ and a torus class $[\mathfrak t_{\delta_\cY}] \in H_k (\cM_{g,n}, \Q)$.

\begin{proposition} \label{prop:abelian} 
The torus class $[\mathfrak t_{\delta_\cY}] \in H_{k}(\cM_{g,n}; \Q)$ is identified with the abelian cycle $\{ T_1, \ldots, T_k \}$ under the natural isomorphism $H_*(\cM_{g,n}; \Q) \rightarrow H_*(\Mod(S); \Q)$.
\end{proposition}

\begin{proof} Let $B$ be an oriented $k$-manifold, with $f\col B \rightarrow \cM_{g,n}$ a continuous map.  Under the natural identification $H_*(\cM_{g,n}; \Q) \xrightarrow{\sim} H_*(\Mod(S); \Q)$, the class $f_*[B]$ maps to the push forward of the fundamental class of $\pi_1(B)$ under the induced map $\pi_1(B) \rightarrow \Mod(S)$.

Choose an \'{e}tale neighborhood $f\col V \rightarrow \ov \cM_{g,n}$ whose image meets $\cY$, where $V$ is a smooth variety and $f^{-1} (\partial \ov \cM_{g,n})$ is a divisor with simple normal crossings.  Let $z$ be a point in the preimage of $y$.  Consider the induced map from a $k$-torus $T_\epsilon \rightarrow \cM_{g,n}$, as in \S\ref{sec:torus}.  The torus class $[\mathfrak t_{\delta_\cY}]$ is defined to be the push forward of the fundamental class $[T_\epsilon]$, which is identified with the push forward of the fundamental class of $\pi_1(T_\epsilon)$ under the induced map to $\Mod(S)$.  The choice of coordinates identifies $\pi_1(T_\epsilon)$ with the free abelian group $\Z^k$, and we must show that the $i$th basis element maps to the right handed Dehn twist $T_i$.  In other words, we must show that the monodromy action on the universal curve pulled back to a small loop around the boundary divisor $D_i \subset V$ is a right handed Dehn twist along the curve $\gamma_i$ that contracts to a node as one moves toward $D_i$.

Let $Z \subset V$ be a smooth algebraic curve that meets $D_i$ transversally at $z$, and consider a small disc around $z$ in $Z$.  Then a well-known local computation (see, for instance, \cite[\S3.2]{KorkmazStipsicz09}) shows that the pullback of the universal curve to this small disc is a Lefschetz pencil, the vanishing cycle is $[\gamma_i]$, and the monodromy action on a smooth fiber is the right handed Dehn twist $T_i$ along the vanishing cycle, as required.
\end{proof}

Having identified these abelian cycles with torus classes in $\cM_{g,n}$, we can describe all relations among them in terms of coboundaries on the tropical moduli space.

\begin{theorem} \label{thm:abelianrelations} A linear combination of abelian cycles in $H_{3g-3+n}(\cM_{g,n}; \Q)$ vanishes if and only if the corresponding cellular $(3g-4+n)$-cochain on $\Dgn$ is a coboundary.
\end{theorem}

\begin{proof} This is given by the special case of Corollary~\ref{cor:relations} where $\cX = \cM_{g,n}$ and $\ocX$ is the Deligne--Mumford compactification $\overline \cM_{g,n}$.
\end{proof}

\section{Contractibility of the repeated marking subcomplex} \label{sec:contractibility}

We now return to the combinatorial topology of $\Dgn$ and proceed with the proof of Theorem~\ref{thm:contractible}.
Let $\rep \subseteq \Dgn$ be the subcomplex parametrizing volume 1 tropical curves $(G,\ell,m,w)$ in which the marking function $m\colon \{1,\ldots,n\}\to V(G)$ is not injective,~i.e., some vertex has at least two of the marked points. We say that such a curve has {\em repeated markings,} and similarly we say that $\Gmw=(G,m,w)\in J_{g,n}$ is {\em repeating} in this situation.  The main result of this section, restated from the introduction, is the following.

{ \renewcommand{\thetheorem}{\ref{thm:contractible}}
\begin{theorem} For all $g>0$ and $n>1$, the repeated marking subcomplex $\rep$ is contractible.
 \end{theorem} \addtocounter{theorem}{-1} }

\begin{proof} 
Recall from \S\ref{subsec:J-g-n} that the {\em core} of $\Gmw = (G,m,w) \in J_{g,n}$ is the smallest connected subgraph of $G$ that contains all cycles and all vertices of positive weight, and that when $g>0$, the core is necessarily nonempty.  

Let $\Delta_i$ be the subcomplex of $\rep$ parametrizing volume 1 curves with at most $i$ core edges.  So $\rep = \Delta_{3g-3+n} \supseteq \cdots \supseteq \Delta_0$.  To prove the theorem, we will define, for each $i>0$, a map
\[ 
\rho_i\col \Delta_i \times[0,1] \rightarrow\Delta_{i}
\] 
that we will show is a strong deformation retraction onto $\Delta_{i-1}$.  After showing that $\rho_i$ is a deformation retract, we will prove that $\Delta_0$ is contractible.  Finally, applying each retraction $\rho_i$ in turn will give the theorem.

Let  $\Gmw = (G,m,w)\in J_{g,n}$ with core $C$.  Then $G-E(C)$ is a disjoint union of trees $Y_v$, each meeting the core at a single vertex $v$.  Say that a core vertex $v \in V(G)$ {\em supports} a marked point $\alpha\in\{1,\ldots,n\}$ if $m(\alpha)\in Y_v$.  In other words, the vertices $m(\alpha)$ and $v$ are in the same connected component $Y_v$ of $G-E(C)$.  
Then we make the following structural observation: the tropical curves in $\rep$ are precisely those in which some core vertex supports more than one marking: this follows from the fact that every vertex of $Y_v$ besides $v$ has weight 0, so any leaf of $Y_v$ has at least two markings.

Now for each $i>0$, we will describe $\Delta_i$ as follows.  First, let $T_i$ be the set of $\Gmw=(G,m,w) \in J_{g,n}$ such that:
 \begin{itemize}
 \item
 $G$ has {\em exactly} $i$ core edges, 
 \item $m$ is not injective, and 
 \item every core vertex $v$ of $G$ that supports more than one marked point is incident to exactly one non-core edge.  We call that edge a {\em distinguished bridge}.  
 \end{itemize}
For each $\Gmw \in T_i$, let $\sigma^1(\Gmw)$ denote the simplex $$\sigma^1(\Gmw)=\{\ell\colon E(G)\rightarrow \R_{\ge0}: \sum \ell(e) = 1\},$$ and write $\sim$ for the equivalence relation on $\coprod_{\Gmw \in T_i}{\sigma^1(\Gmw)}$ whose classes are fibers of the map to $\Dgn$.  Then we have
\[ \Delta_i \cong \Big (\coprod_{\Gmw\in T_i}{\sigma^1(\Gmw)} \Big)/\sim.
\] Indeed, the natural map $\big(\big(\coprod_{\Gmw \in T_i}{\sigma^1(\Gmw)}\big)/ \!\! \sim \! \! \big )  \rightarrow \Dgn$ is a continuous injection from a compact to a Hausdorff space, so it is a homeomorphism onto its image.  In fact, that image is exactly $\Delta_i$; this follows from the observation that any tropical curve with less than $i$ core edges can be obtained by contractions from a graph with exactly $i$ core edges.

We now define a continuous map
\[ \rho_\Gmw\colon \sigma^1(\Gmw) \times [0,1] \rightarrow \sigma^1(\Gmw) \rightarrow \Delta_i,
\]
which we will soon use to define the desired map $\rho_i\col \Delta_i \times[0,1] \rightarrow\Delta_{i}.$

 Let $\ell\col E(G)\rightarrow \R_{\ge 0}$ be a point in $\sigma^1(\Gmw)$ and $t\in[0,1]$.  Let $c$ be the length of the shortest core edge (it is possible that $c=0$).  Let $j>0$ be the number of core vertices that support more than one marked point.  We define $\ell_t=\rho_\Gmw(\ell, t):E(G)\rightarrow \R_{\ge 0}$ by sending
\[ \ell_t(e) = \begin{cases} \ell(e)-tc &\text{if $e$ is a core edge} \\ \ell(e)+tci/j &\text{if $e$ is a distinguished bridge} \\ \ell(e) &\text{else.}
\end{cases}
\] Then $\rho_\Gmw$ is evidently continuous. Note that if $c=0$ then $\ell_t = \ell$ for all $t$.  In other words, on the level of tropical curves, any curve with fewer than $i$ edges in its core is fixed by $\rho_\Gmw$ for all $t$; while $\rho_\Gmw$ sends a curve with exactly $i$ edges in its core to the one obtained by shortening its core edges by $ct$ and apportioning the extra length equally among each of the $j$ distinguished bridges (which, for the purpose of this description, are allowed to have zero length to begin with).  In other words, the map
$$\coprod{\rho_\Gmw}\colon \coprod{\sigma^1(\Gmw)} \times [0,1] \rightarrow \Delta_i$$
respects $\sim$, and so descends to a continuous map $\rho_i\colon \Delta_i\times [0,1]\rightarrow \Delta_i$ which by its construction is evidently a strong deformation retract onto $\Delta_{i-1}$.

So it remains only to prove the claim that $\Delta_0$ is contractible.  Let $T$ denote the set of objects $\Gmw\in J_{g,n}$ in which there is a single vertex $v_0$ of weight $g$ which is 1-valent and has no marked points on it.  Then
\[ \Delta_0 \cong \Big(\coprod_{\Gmw\in T} \sigma^1(\Gmw)\Big)/\sim,
\] 
where again $\sim$ denotes the equivalence relation whose classes are fibers of the canonical map to $\Delta_{g,n}$.
This is because any combinatorial type of tropical curve appearing in $\Delta_0$ may be recovered from some $\Gmw\in T$ by contracting the unique edge incident to the vertex $v_0$ of weight $g$. Now for each $\Gmw = (G,m,w)\in T$, define a continuous map $\rho_\Gmw\col \sigma^1(\Gmw)\times[0,1]\rightarrow \sigma^1(\Gmw) \rightarrow\Delta_0$ as follows.  Let $\ell\col E(\Gmw)\rightarrow \R_{\ge 0}$ be a point in $\sigma^1(\Gmw)$ and $t\in[0,1]$, and let $e_0$ be the unique edge incident to the weight $g$ vertex in $G$. Then we define $\ell_t=\rho_\Gmw(\ell, t)\colon E(\Gmw)\rightarrow \R_{\ge 0}$ by
\[ \ell_t(e) = \begin{cases} (1-t)\ell(e)+t &\text{if $e = e_0$} \\ (1-t)\ell(e)&\text{else.}
\end{cases}
\] This is again evidently a continuous map, and the map $\coprod \rho_\Gmw$ again descends to a continuous map $\rho_0\colon \Delta_0\times [0,1]\to \Delta_0$ which is a deformation retraction onto a point in $\Delta_0$.  Namely, the retraction is onto the tropical curve that has a single bounded edge of length 1, between a vertex of weight $g$ and a vertex of weight $0$ with all $n$ markings.
\end{proof}

\section{Proofs of Theorems~\ref{thm:boundary}, \ref{thm:connected}, \ref{thm:cohomology}, and \ref{thm:dualbasis}} \label{sec:proofs}

In this section, we apply Theorem~\ref{thm:contractible}, which identifies $\rep$ as a large subcomplex in $\Delta_{g,n}$ that is contractible, to prove the rest of the theorems stated in the introduction.  We begin by showing that contracting $\Delta^{\mathrm{rep}}_{1,n}$ produces a bouquet of $(n-1)!/2$ spheres indexed by cyclic orderings of the set $\{1,\ldots,n\}$.

{ \renewcommand{\thetheorem}{\ref{thm:boundary}}
\begin{theorem} Both $\Delta_{1,1}$ and $\Delta_{1,2}$ are contractible.  For $n \geq 3$, the tropical moduli space $\Dn$ is homotopy equivalent to a wedge sum of $(n-1)! / 2$ spheres of dimension $n-1$.
 \end{theorem} \addtocounter{theorem}{-1} }

\begin{proof} The core of a genus $1$ tropical curve is either a single vertex of weight 1 or a cycle.  Now if $\Gamma \in \Dn \setminus \Delta^{\mathrm{rep}}_{1,n}$ then it cannot possibly contain a weight 1 vertex; if it did, then the underlying graph of $\Gamma$ is a tree, and either it has just one vertex $v_0$ with all $n$ markings, or else and any leaf distinct from $v_0$ has more than one marking on it.  

So the core of $\Gamma$ is a cycle with all vertices of weight zero, and each vertex supports at most one marked point.  The stability condition then ensures that each vertex supports exactly one marked point.  In other words, the combinatorial types of tropical curves that appear outside the repeated marking locus consist of an $n$-cycle with the markings $\{1,\ldots,n\}$ appearing around that cycle in a specified order.  There are $(n-1)!/2$ possible orders $\tau$ of $\{1,\ldots,n\}$ up to symmetry of the $n$-cycle, so we have $(n-1)!/2$ such combinatorial types $\Gmw_\tau.$

Now, the case $n=2$ is special: in this case, the unique cell of $\Delta_{1,2}$ not in $\Delta^{\mathrm{rep}}_{1,2}$ consists of two vertices and two edges between them.  Swapping the edges gives a nontrivial $\Z/2\Z$ automorphism on this cell, which then retracts to the repeated marking locus.  So $\Delta_{1,2}$ is contractible by Theorem~\ref{thm:contractible}.  (The case $n=1$ is even easier, as $\Delta_{1,1}$ is just a point.)

For $n > 2$, each $\Gmw_\tau$ has no  nontrivial automorphisms, so the image of the interior of $\sigma^1(\Gmw_\tau)$ in $\Dn$ is an $(n-1)$-disc whose boundary is identified with the repeated marking locus.  Now it follows again from Theorem~\ref{thm:contractible} that $\Dn$ has the homotopy type of a wedge of $(n-1)!/2$ spheres of dimension $n-1$.  
\end{proof}

Next, we prove a connectivity bound for $\Dgn$.  The key technical input is Theorem~\ref{thm:contractible}(1).

{ \renewcommand{\thetheorem}{\ref{thm:connected}}
\begin{theorem} 
For $g>1$, the tropical moduli space $\Dgn$ is $(n - 5g + 4)$-connected.
 \end{theorem} \addtocounter{theorem}{-1} }

\noindent Note that the bound given by Theorem~\ref{thm:connected} does not hold when $g$ is $0$ or $1$.  Indeed, Theorem~\ref{thm:boundary} shows that that $\Dn$ is exactly $(n-2)$-connected, for $n \geq 3$.  Similarly, $\Delta_{0,n}$ is exactly $(n-5)$-connected, for $n \geq 5$.  Rational homology vanishing suggests that $\Delta_{g,n}$ may be even more highly connected.  See Proposition \ref{prop:topg}.

%Say $\Gmw\in J_{g,n}$ is {\em bridgy} if there exists some $\Gmw'\in J_{g,n}$ admitting an arrow in $J_{g,n}$ to $\Gmw$, such that $\Gmw'$ has a bridge.  We remark that $\Gmw$ is bridgy if and only if it has a cut vertex or a vertex of positive weight. (A cut vertex is a vertex whose removal topologically disconnects the underlying 1-complex of $\Gmw$.)
%Let $J^\mathrm{rep}_{g,n}$ denote the full subcategory of $J_{g,n}$ consisting of graphs $\Gmw\in J_{g,n}$ with repeated markings.  
Let $J^\mathrm{rep}_{g,n}$ denote the full subcategory of $J_{g,n}$ consisting of graphs $\Gmw\in J_{g,n}$ with repeated markings. %Precisely, recall from \S\ref{subsec:blocks-etc} that $\Gmw$ has an articulation point if and only if it has a cut vertex, a vertex with repeated markings, or a vertex of positive weight.  Then $\Gmw\in J_{g,n}$ if and only if $\Gmw$ has an articulation point. 
Now for every $\Gmw=(G,m,w) \in J_{g,n}\setminus J^\mathrm{rep}_{g,n}$, define $S(\Gmw)\subseteq E(\Gmw)$ to be the set of edges $e\in E(\Gmw)$ satisfying:
\begin{itemize}
\item $e$ has distinct endpoints $x$ and $y$, i.e.~$e$ is not a loop, 
\item both  $x$ and $y$  are once-marked, and  $$w(x)=w(y)=0\text{ and } \val_G(x)=\val_G(y)=2;$$ 
\item no edge is parallel to $e$.
\end{itemize}
Here and below, $\val_G(x)$ denotes the number of half-edges incident to $x$ in the underlying ummarked graph $G$.  It should not be confused with $\val(x)$, introduced in Section 2, which counts the number of half-edges plus the number of marked points at $x$ in the weighted, marked graph $\Gmw$.

\begin{lemma}\label{lem:n-5g+5} %Let $g\ge 1$, and $(g,n)\ne (1,1)$ or $(1,2)$.  
Let $g>1$. Then for any $\Gmw\in J_{g,n}\setminus J^\mathrm{rep}_{g,n}$, we have:
\begin{enumerate}[(i)]
\item For every $\rho\in \operatorname{Aut}(\Gmw)$ and $e\in S(\Gmw)$, we have $\rho\cdot e = e$.
\item If $e,f\in E(\Gmw)\setminus S(\Gmw)$ are distinct, and $\Gmw/e\not\in J^\mathrm{rep}_{g,n}$, then $\overline{f}\not \in S(\Gmw/e)$, where $\overline{f}$ denotes the image of $f$ in $\Gmw/e$.
\item $|E(\Gmw)|> |S(\Gmw)|\ge n-5g+5.$
\end{enumerate}
\end{lemma}
\begin{proof}
Statement (i) follows from the fact that $e$ is the only edge with the marked vertices $x$ and $y$ as its endpoints.  For statement (ii), we have that $f\not \in S(\Gmw),$ which is equivalent to at least one of the following conditions holding: 
\begin{enumerate}
\item $f$ is a loop, 
\item $f=xy$ is a nonloop with a parallel edge $f'$,
\item $f = xy$ is a nonloop with $w(x)>0$ or $w(y)>0$, or 
\item $f = xy$ is a nonloop with $\val_G(x)\ge 3$ or $\val_G\ge 3$.
\end{enumerate}
Moreover, if (1), (3), or (4) holds for $f$ then (1), (3), or (4) holds for $\ov{f}\in E(\Gmw/e)$ as well. If (2) holds for $f$ then (1) or (2) holds for $\ov{f}\in E(\Gmw/e)$ as well. So we conclude that $\ov{f}\not\in S(\Gmw/e)$.

For statement (iii), fix $g>1$ and fix $\Gmw=(G,m,w) \in J_{g,n}\setminus J^\mathrm{rep}_{g,n}$. First note that $S(\Gmw)\subsetneq E(\Gmw)$, since if $S(\Gmw)=E(\Gmw)$ then every vertex of $\Gmw$ would be 2-valent, once-marked and weight zero, which is impossible for $g>1$.  Next, since $\Gmw\not\in J^\mathrm{rep}_{g,n}$, we have that $m$ is injective.  Thus $\Gmw$ is obtained from some $\Gmw_0\in J_{g,0}$ by adding $n$ markings at distinct points, either on vertices of $\Gmw_0$ or interiors of edges of $\Gmw_0$. More formally, the latter operation is regarded as an introduction of a new weight zero vertex supporting that marking. At most $2g-2$ markings appear on vertices of $\Gmw_0$, so at least $n-2g+2$ markings appear on edges of $\Gmw_0$. Let $n_1,\ldots,n_k>0$ be the collection of positive numbers of marked points added on edges.  Note $k\le |E(\Gmw_0)|\le 3g-3$. Then we have that $|S(\Gmw)| = \sum_{i=1}^{k} (n_i-1) \ge (n-2g+2)-(3g-3) = n-5g+5.$
\end{proof}
 
Before proving Theorem~\ref{thm:connected}, we fix notation for simplices and their barycentric subdivisions.  
First, for $E$ any finite set, write 
$$\Delta^E = \{\ell\colon E\rightarrow \R_{\ge 0}~|~\sum \ell(e)=1\} \subset \R^E.$$
Now given any subset $Z\subseteq E$ we may regard $\Delta^Z $ as a face of $\Delta^E$. 
Write $Z^c = E\setminus Z$.
Given $\Delta_1\subseteq \Delta^Z$ and $\Delta_2 \subseteq \Delta^{Z^c}$, we write $\Delta_1\ast \Delta_2 = \operatorname{conv}(\Delta_1,\Delta_2)\subseteq \Delta^E$ for the convex hull; note that $\Delta_1\ast \Delta_2$ is naturally identified with the join of $\Delta_1$ and $\Delta_2$.  For example, $\Delta^E = \Delta^Z \ast \Delta^{Z^c}$.  For any simplex $\Delta$, write $\operatorname{Vert}(\Delta)$ for the vertex set of $\Delta$, and for $v\in\operatorname{Vert}(\Delta)$, write $\Delta_v$ for the unique facet of $\Delta$ not containing $v$.

For any simplex $\Delta$, let $\Delta^\circ$ denote the interior of $\Delta$ and let $\partial\Delta$ denote the boundary of $\Delta$.  If $\dim\Delta=0$, we take $\Delta^\circ = \Delta$ and $\partial\Delta=\emptyset$.  Let $\operatorname{Bar}(\Delta)$ denote the barycentric subdivision of $\Delta$. Thus $\operatorname{Bar}(\Delta)$ is a collection of closed simplices forming a polyhedral complex supported on $\Delta$.  Let $c(\Delta)$ denote the central vertex of $\operatorname{Bar}(\Delta)$; that is, $c(\Delta)$ is the unique simplex of $\operatorname{Bar}(\Delta)$ which does not meet $\partial\Delta$.  

Let $\operatorname{Bar}^\circ(\Delta)$ denote the subcollection of $\operatorname{Bar}(\Delta)$ consisting of those simplices containing $c(\Delta)$ as a vertex. Note that the members of $\operatorname{Bar}^\circ(\Delta)$ are still closed simplices, but $\operatorname{Bar}^\circ(\Delta)$ is no longer a polyhedral complex, since not all faces of simplices in $\operatorname{Bar}^\circ(\Delta)$ are in $\operatorname{Bar}^\circ(\Delta)$.  Then we observe the equality of sets 
\begin{equation}\label{eq:deltaint}
\coprod_{\Delta'\in \operatorname{Bar}^\circ(\Delta)}  \!\! (\Delta')^\circ = \Delta^\circ.
\end{equation}
 
\begin{proof}[Proof of Theorem~\ref{thm:connected}] For any positive $g$, we note that $\Delta_{g,n}$ is homotopy equivalent to $\Delta_{g,n}/\Delta_{g,n}^\mathrm{rep}$, using Theorem~\ref{thm:contractible}.  We will give $\Delta_{g,n}/\Delta_{g,n}^\mathrm{rep}$ a CW-complex structure with a single $0$-cell and all of whose positive dimensional cells have dimension at least $n-5g+5$, which is enough to prove the theorem.

We fix a single $0$-cell denoted $\bullet_{g,n}^\mathrm{rep}$ to be the image of $\Delta_{g,n}^\mathrm{rep}$ in $\Delta_{g,n}/\Delta_{g,n}^\mathrm{rep}$.  Now, let $\Gmw \in J_{g,n}\setminus J_{g,n}^\mathrm{rep}$.  Recall the definition of $S(\Gmw)$ preceding Lemma~\ref{lem:n-5g+5}; we will write $S=S(\Gmw)$ and $T = E(\Gmw)\setminus S$. Note that $T\ne\emptyset $ is guaranteed by Lemma~\ref{lem:n-5g+5}(3). Let $\Sigma(\Gmw)$ be the collection of simplices 
$$\Sigma(\Gmw)=\{\Delta^S\} \ast \operatorname{Bar}^\circ(\Delta^{T}) = \{\Delta^S\ast \Delta': \Delta'\in  \operatorname{Bar}^\circ(\Delta^{T})\}.$$  
Notice every simplex of $\Sigma(\Gmw)$ sits inside $\Delta^{E(\Gmw)}$.  We also record for future use the equality of sets
\begin{equation}\label{eq:interior}
(\Delta^{E(\Gmw)})^\circ = \coprod_{\Delta''\in \Sigma(\Gmw)} (\Delta'')^\circ.
\end{equation}

Now $\Aut(\Gmw)$ naturally acts on the points of $\Delta^{E(\Gmw)}$; write $\cdot_\Delta$ for this action.  We now claim the following.

\begin{claim}\label{claim:pointwise}
The action $\cdot_\Delta$ of $\Aut(\Gmw)$ on the points of  $\Delta^{E(\Gmw)}$ induces an action of $\Aut(\Gmw)$ on the finite set $\Sigma(\Gmw)$, which we denote $\cdot_\Sigma$.  Then if $\rho \cdot_\Sigma \Delta' = \Delta'$ for some $\rho \in \Aut(\Gmw)$ and some $\Delta'\in \Sigma(\Gmw)$, then $\rho\in\Aut(\Gmw) \cdot_\Delta p = p$  for each point $p\in \Delta'$.  In other words, if $\rho$ fixes a simplex $\Delta'$ setwise, then $\rho$ fixes $\Delta'$  pointwise.
\end{claim}
\begin{proof}[Proof of Claim~\ref{claim:pointwise}]
By Lemma~\ref{lem:n-5g+5}(1), the action $\cdot_\Delta$ of $\Aut(\Gmw)$ on the points of $\Delta^{E(\Gmw)}$
restricts to two actions on $\Delta^S$ and $\Delta^T$. The first action is trivial, while the second induces an action of $\Aut(\Gmw)$ on the finite set $\operatorname{Bar}^\circ(\Delta^T)$.  This latter action has the property that if $\rho\in\Aut(\Gmw)$ and $\rho\cdot\Delta'=\Delta'$ for $\Delta'\in \operatorname{Bar}^\circ(\Delta^T)$, then $\rho$ fixes $\Delta'$ pointwise under $\cdot_\Delta$, by properties of the barycentric subdivision.  Indeed, given a simplex in $\operatorname{Bar}^\circ(\Delta^T)$, every vertex lies in the interior of a unique face of $\Delta^T$ and all of those faces have dimensions that are different from each other.  Together these facts imply the claim.  
\end{proof}

Now let $\{\Delta^\Gmw_\alpha\}$ be a set of $\Aut(\Gmw)$-orbit representatives of $\Sigma(\Gmw)$.  For an element $\Delta^\Gmw_\alpha$ we let $\phi^\Gmw_\alpha$ be the composition of the natural maps
$$\Delta^\Gmw_\alpha \hookrightarrow \Delta^{E(\Gmw)} \rightarrow \Delta_{g,n}\rightarrow \Delta_{g,n}/\Delta_{g,n}^\mathrm{rep}.$$
Let $e^\Gmw_\alpha = \phi((\Delta^\Gmw_\alpha)^\circ).$  The following claim will finish the proof of Theorem~\ref{thm:connected}.

\begin{claim}\label{claim:cw}
As $\Gmw$ ranges over $J_{g,n}\setminus J_{g,n}^\mathrm{rep}$ and $\Delta^\Gmw_\alpha$ ranges over a set of $\Aut(\Gmw)$-orbit representatives for $\Sigma(\Gmw)$, the spaces $e^\Gmw_\alpha$, taken together with the $0$-cell $\bullet_{g,n}^\mathrm{rep}$, give a CW structure on $\Dgn/\Dgn^\mathrm{rep}$, with characteristic maps $\phi^\Gmw_\alpha\colon \Delta^\Gmw_\alpha \rightarrow \Dgn/\Dgn^\mathrm{rep}$, such that every cell except for $\bullet_{g,n}^\mathrm{rep}$ has dimension at least $n-5g+5$.
\end{claim}

\begin{proof}[Proof of Claim~\ref{claim:cw}]  We verify the following statements:
\begin{enumerate}[(i)]
\item For each $\Gmw\in J_{g,n}\setminus J_{g,n}^\mathrm{rep}$, each $\phi^\Gmw_\alpha|_{(\Delta^\Gmw_\alpha)^\circ}$ is a homeomorphism $(\Delta^\Gmw_\alpha)^\circ\xrightarrow{\cong} e^\Gmw_\alpha$, and $\dim e^\Gmw_\alpha \ge n-5g+5$.
\item $\Dgn/\Dgn^\mathrm{rep}=\{\bullet_{g,n}^\mathrm{rep}\}\amalg \coprod e^\Gmw_\alpha$ as sets.
\item Writing $(\Dgn/\Dgn^\mathrm{rep})^{(k)}$ for the union in $\Dgn/\Dgn^\mathrm{rep}$ of all cells of dimension at most $k$, 
we have $$\phi^\Gmw_\alpha(\partial \Delta^\Gmw_\alpha)\subseteq (\Dgn/\Dgn^\mathrm{rep})^{(\dim \Delta^\Gmw_\alpha - 1)}.$$
\end{enumerate}

{\em Proof of (i).}  Given $\Gmw\in J_{g,n}\setminus J_{g,n}^\mathrm{rep}$ and given any $\Delta^\Gmw_\alpha$, the map $\phi^\Gmw_\alpha\colon \Delta^\Gmw_\alpha\rightarrow\Dgn$ admits a quotient  factorization through 
\begin{equation}\label{eq:deltasim}
\Delta^\Gmw_\alpha/\!\sim\,\longrightarrow\Dgn/\Dgn^\mathrm{rep},
\end{equation}
where $\ell\sim\ell'$ if and only if $\phi^\Gmw_\alpha(\ell) = \phi^\Gmw_\alpha(\ell')$.  The map~\eqref{eq:deltasim} is a homeomorphism onto its image, being continuous from a compact space to a Hausdorff space.  To prove (i), it remains to verify that $\sim$ is trivial on $(\Delta^\Gmw_\alpha)^\circ$.  But this was verified in Claim~\ref{claim:pointwise}.  Indeed, if $\ell,\ell'\in (\Delta^\Gmw_\alpha)^\circ$, then $\ell\sim\ell'$ if and only if there exists $\rho\in\Aut(\Gmw)$ with $\rho(\ell)=\ell'$.  In particular by the disjointness of the sets in~\eqref{eq:deltaint}, $\rho$ sends $\Delta^\Gmw_\alpha$ to $\Delta^\Gmw_\alpha$, hence fixes it pointwise by Claim~\ref{claim:pointwise}.  Finally, note that $(\Delta^\Gmw_\alpha)$ is a simplex with at least $|S|+1$ vertices, so $\dim e^\Gmw_\alpha \ge |S| \ge n-5g+5.$

{\em Proof of (ii).} We want to show that the natural map
$$\coprod_{\Gmw\in J_{g,n}\setminus J_{g,n}^\mathrm{rep}}\left(\coprod_\alpha (\Delta^\Gmw_\alpha)^\circ \right)\longrightarrow (\Dgn/\Dgn^\mathrm{rep})-\{\bullet_{g,n}^\mathrm{rep}\}$$ 
yields an equality of sets.  It suffices to show that for each $\Gmw\in J_{g,n}\setminus J_{g,n}^\mathrm{rep}$, the natural map $$\coprod_\alpha (\Delta^\Gmw_\alpha)^\circ \rightarrow (\Delta^{E(\Gmw)})^\circ/\Aut(\Gmw)$$ 
yields an equality of sets.  But this follows from Equation~\eqref{eq:interior} and the fact that the $\{\Delta^\Gmw_\alpha\}$ was a choice of $\Aut(\Gmw)$-orbit representatives for $\Sigma(\Gmw)$. 

{\em Proof of (iii).} We proceed by induction on $\dim \Delta^\Gmw_\alpha$.  Given $\Gmw\in J_{g,n}\setminus J_{g,n}^\mathrm{rep}$ and given any $\Delta^\Gmw_\alpha$, it suffices to check that for every vertex $v$ of $\Delta^\Gmw_\alpha$, we have
\begin{equation}\label{eq:facetdims}
\phi^\Gmw_\alpha((\Delta^\Gmw_\alpha)_v) \subseteq (\Dgn/\Dgn^\mathrm{rep})^{(\dim \Delta^\Gmw_\alpha - 1)}.
\end{equation} Recall that $(\Delta^\Gmw_\alpha)_v$ denotes the unique facet of $\Delta^\Gmw_\alpha$ not containing $v$.

Recall the definitions of $S=S(\Gmw)$ and $T=E(\Gmw)\setminus S(\Gmw)$.  Recall that Lemma~\ref{lem:n-5g+5}(3) implies $T\ne \emptyset$.  Write $\Delta^\Gmw_\alpha = \Delta^S \ast \Delta'$ for $\Delta'\in \operatorname{Bar}^\circ(\Delta^T)$.  Then $\dim \Delta^\Gmw_\alpha = |S|+\dim \Delta'$. Let $\Delta'' = \Delta'_{c({\Delta^T})}$.  
Note that 
\begin{equation}\label{eq:vertexcases}
\operatorname{Vert}(\Delta^\Gmw_\alpha) = \operatorname{Vert}(\Delta^S) \amalg \operatorname{Vert}(\Delta'')\amalg \{c({\Delta^T})\}.
\end{equation}

Let $v\in \operatorname{Vert}(\Delta^\Gmw_\alpha)$.  We check the claim in~\eqref{eq:facetdims} in each of the three cases~\eqref{eq:vertexcases}.   First, if $v\in\operatorname{Vert}(\Delta^S)$ then we have $\phi^\Gmw_\alpha((\Delta^\Gmw_\alpha)_v) = \{\bullet_{g,n}^{\mathrm{rep}}\}$.  This is because for any $e\in S(\Gmw)$, we have $\Gmw/e\in J_{g,n}^\mathrm{rep}.$

Second, if $v\in\operatorname{Vert}(\Delta'')$ then $(\Delta^\Gmw_\alpha)_v=\Delta^S\ast \Delta'_v$ is $\Aut(\Gmw)$-equivalent to some $\Delta^\Gmw_{\alpha'}$.  Clearly $\dim \Delta^\Gmw_{\alpha'} < \dim \Delta^\Gmw_{\alpha}$, so $$\phi^\Gmw_\alpha((\Delta^\Gmw_{\alpha})_v) \subseteq \Delta_{g,n}^{(\dim \Delta^\Gmw_{\alpha}-1)}$$ as desired.

Finally, suppose $v = c({\Delta^T})$.  %Now if $\Delta'=c(\Delta^T)$ then $(\Delta^\Gmw_\alpha)_v=\Delta^S$ and $\phi^\Gmw_\alpha(\Delta^S) = \{\bullet_{g,n}^\mathrm{br}\}$.
Let $T''\subseteq T$ be such that $\Delta''\in \operatorname{Bar}^\circ(\Delta^{T''})$; note $T''$ is defined uniquely by this property.  (If $\Delta''=\emptyset$ then we take $T''=\emptyset$.)  Note also that $T''\ne T$.  Let $T'=T\setminus T''$, so $T'\ne\emptyset$.
 Let $\Gmw' = \Gmw/T'$; since $T'\ne\emptyset$, we regard $\Delta^{E(\Gmw')}$ as a proper face of $\Delta^{E(\Gmw)}$.  
 Now if $\Gmw'\in J_{g,n}^\mathrm{rep}$, then we have $\phi^\Gmw_\alpha((\Delta^\Gmw_{\alpha})_v) =\{\bullet_{g,n}^\mathrm{rep}\}$ and the claim is proved. Otherwise, let $S'=S(\Gmw')$.  By Lemma~\ref{lem:n-5g+5}(2), we may regard $S'$ as a subset of $S$.  Write $S''=S\setminus S'$.  So $E(\Gmw')\setminus S' = S''\cup T''$.  

We have
\begin{eqnarray*}
(\Delta^\Gmw_\alpha)_v &=& \Delta^S\ast \Delta'' \\
&=& \Delta^{S'}\ast(\Delta^{S''}\ast\Delta'').
\end{eqnarray*}
Now  $\Delta^{S''}\ast\Delta''\subset \Delta^{S''\cup T''}$ is a simplex, of dimension $\dim\Delta''+|S''|$.  Moreover, since $\Delta''\in\operatorname{Bar}^\circ(\Delta^{T''})$, it follows that $\Delta^{S''}\ast\Delta''\subset \Delta^{S''\cup T''}$ is a union of cells in $\operatorname{Bar}^\circ(\Delta^{S''\cup T''})$, each of dimension at most $\dim\Delta''+|S''|$.  Recalling that $S''\cup T''=E(\Gmw')\setminus S'$, it follows that $\Delta^{S'}\ast(\Delta^{S''}\ast\Delta'')$ is a union of a subset of cells $\Sigma'\subseteq \Sigma(\Gmw')$, each cell in $\Sigma'$ having dimension at most $\dim\Delta''+|S''|+|S'|=\dim\Delta''+|S|$.  Moreover, as $\Delta^{\Gmw'}_{\alpha'}$ runs over a set of $\Aut(\Gmw)$-orbit representatives of the cells in $\Sigma'$, we have $\dim \Delta^{\Gmw'}_{\alpha'}< \dim \Delta^\Gmw_\alpha$, so by induction $\phi^{\Gmw'}_{\alpha'}(\Delta^{\Gmw'}_{\alpha'}) \subseteq (\Dgn/\Dgn^\mathrm{rep})^{(\dim\Delta''+|S|)}.$  Finally, we have $\dim\Delta''+|S| = \dim\Delta'+|S|-1 = \dim\Delta^\Gmw_\alpha -1$ so statement (iii) is proved. 
\end{proof}

Having proved Claim~\ref{claim:cw}, the proof of Theorem~\ref{thm:connected} is complete.
\end{proof}

We note the following immediate consequence of Theorem~\ref{thm:connected}.

\begin{corollary} For $g>1$, the reduced homology $\widetilde H_k (\Dgn;\Z)$ vanishes for $k\le n-5g+4$.
\end{corollary}

\noindent Thus $\widetilde H_k (\Dgn;\Z)$ is concentrated in the top $8g-9$ degrees.  With rational coefficients, vanishing in an even wider range can be seen using Harer's computation of the virtual cohomological dimension of $\cM_{g,n}$ from \cite{Harer86}, as follows.

\begin{proposition} \label{prop:topg} The reduced rational homology of $\Dgn$ is supported in the top $g - \delta_{0,n}$ degrees, for $g \geq 1$.
\end{proposition}

\begin{proof}
Suppose $g \geq 1$.  By \cite{Harer86}, the virtual cohomological dimension of $\cM_{g,n}$ is $4g - 4 + n - \delta_{0,n}$, where $\delta_{ij}$ is the Kronecker delta function.  Therefore, the top weight cohomology of $\cM_{g,n}$ is supported in degrees $\{ 3g - 3 + n, \ldots, 4g - 4 + n - \delta_{0,n} \}$, and hence $\widetilde H_k(\Dgn, \Q)$ vanishes unless $2g - 3 + n + \delta_{0,n} \leq k \leq  3g - 4 + n$.
\end{proof}

Now we verify Theorem~\ref{thm:cohomology}, which gives the top weight rational cohomology of $\cM_{1,n}$ as an $S_n$-representation.  The result is:

{ \renewcommand{\thetheorem}{\ref{thm:cohomology}}
\begin{theorem} For each $n\ge 1$, the top weight cohomology of $\cM_{1,n}$ is
 \[ \Gr_{2n}^W H^i(\cM_{1,n}; \Q) \cong \left \{ \begin{array}{ll} \Q^{(n-1)!/2}
 & \mbox{ \ \ for $n \geq 3$ and $i = n$,} \\
 												0 & \mbox{ \ \ otherwise.} \end{array} \right.
 \] Moreover, for each $n\ge 3$, the representation of $S_n$ on $\Gr^W_{2n} H^n(\mathcal{M}_{1,n}; \Q)$ induced by permuting marked points is
 $$\mathrm{Ind}_{D_n,\phi}^{S_n} \, \mathrm{Res}^{S_n}_{D_n,\psi} \, \mathrm{sgn}.$$ 
 \end{theorem}
  \noindent Here $\phi\colon D_n \rightarrow S_n$ is the dihedral group of order $2n$ acting on the vertices of an $n$-gon, $\psi\colon D_n \rightarrow S_n$ is the corresponding embedding into the permutation group on the {\em edges} of the $n$-gon, and $\mathrm{sgn}$ denotes the sign representation of $S_n$.
\addtocounter{theorem}{-1} }

\begin{proof} The first statement follows from combining Theorem~\ref{thm:boundary}, Corollary~\ref{cor_dual_complex}, and Proposition~\ref{prop:topweight}.
Now let $n\ge 3$.  It suffices to identify the representation of $S_n$ on the vector space
$$V = H_{n-1}(\Dn/\Delta_{1,n}^{\mathrm{rep}}; \Q)$$
obtained by permuting the marked points.  Indeed, the compactification $\mathcal{M}_{g,n} \subset \overline{\mathcal{M}}_{g,n}$ and the identification of $\Delta_{g,n}$ with the dual complex of the Deligne--Mumford boundary divisor $\Delta(\ocM_{g,n}\smallsetminus \cM_{g,n})$ are certainly equivariant under permuting marked points, as is the contraction of $\Dgn^{\mathrm{rep}}$ studied in Theorem~\ref{thm:contractible}.  Recall that $V$ is the $\Q$-vector space spanned by the homology classes of the $(n-1)$-spheres in the wedge $\Dn/\Delta_{1,n}^{\mathrm{rep}}$, which are in bijection with cyclic orderings of $\{1,\ldots,n\}$, by Theorem~\ref{thm:boundary}.

Let $\phi\colon D_n \rightarrow S_n$ be the embedding of the dihedral group as a subgroup of the permutations of the vertices $\{1,\ldots,n\}$ of an $n$-cycle. Choose left coset representatives $\sigma_1,\ldots,\sigma_k$, where $k={(n-1)!/2}$, and write $[\sigma_i]$ for the corresponding basis elements of $V$.  For any $\pi \in S_n$, we have $\pi \sigma_i = \sigma_j \pi'$ for some $\pi'\in D_n$. Then $\pi\cdot[\sigma_i] = \pm [\sigma_j]$, where the sign depends exactly on the sign of the permutation on the {\em edges} of the $n$-cycle induced by $\pi'$. This is because the ordering of the edges determines the orientation of the corresponding sphere in $\Dn/\Delta_{1,n}^{\mathrm{rep}}$.  Therefore the representation of $S_n$ on $V$ is exactly $\mathrm{Ind}_{D_n,\phi}^{S_n}\mathrm{Res}_{D_n,\psi}^{S_n} \mathrm{sgn}$, where the restriction is according to the embedding of $\psi\colon D_n\rightarrow S_n$ into the group of permutations of edges of the $n$-cycle.
\end{proof}

We conclude with the proof of Theorem~\ref{thm:dualbasis}, giving an explicit dual basis of torus classes to the top-weight cohomology of $\cM_{1,n}$.

{ \renewcommand{\thetheorem}{\ref{thm:dualbasis}}
\begin{theorem} For each $n\ge 3$, the classes $[\mathfrak{t}_\Gmw] \in H_n(\cM_{1,n}; \Q)$ indexed by cycles with $n$ vertices labeled by distinct elements of $\{1, \ldots, n\}$ form a dual basis to $\Gr_{2n}^W H^n(\cM_{1,n}; \Q)$.
 \end{theorem} \addtocounter{theorem}{-1} }

\begin{proof} The proof of Theorem~\ref{thm:boundary} shows that the tropical moduli space $\Dn$ has a contractible subcomplex $\Delta_{1,n}^{\mathrm{rep}}$ whose complement is the union of the $(n-1)!/2$ open cells of dimension $n-1$ corresponding to graphs $\Gmw\in J_{1,n}$ in the set $\mathcal L$ of loops with $n$ vertices labeled by distinct elements of $\{1, \ldots, n \}$.  Contracting this subcomplex gives a homotopy equivalence onto a bouquet of $(n-1) ! / 2$ spheres indexed by $\mathcal L$.  Let $\{e_\Gmw\}_{\Gmw \in \mathcal L}$ be the natural basis for $H_{n-1}(\Delta_{1,n}/\Delta_{1,n}^{\mathrm{rep}}; \Q)$ induced by a choice of orientation on each of these spheres (e.g. by making a choice depending on the cyclic ordering of the markings on $\Gmw$), with $e'_\Gmw$ the corresponding basis for $H_{n-1}(\Delta_{1,n}, \Q)$.

Suppose $F$ is a cellular $(n-1)$-chain on $\Delta_{1,n}$ whose homology class is $[F] = e'_\Gmw$.  Then the oriented face of $\Delta_{1,n}$ corresponding to $\Gmw$ appears in $F$ with multiplicity 1, and the faces corresponding to other loops with $n$ distinct labeled vertices appear with coefficient zero.  Therefore, by Proposition~\ref{prop:cyclesandboundaries}, the natural isomorphism $H_{n-1}(\Dn; \Q) \xrightarrow{\sim} \Gr_{2n}^W H^{n}(\cM_{1,n}; \Q)$ maps $e'_\Gmw$ to the class of a differential form $[\omega_\Gmw]$ such that $\int_{[\mathfrak t_\Gmw]} \omega_\Gmw = (2 \pi \sqrt{-1})^n$ and $\int_{[\mathfrak t_{\Gmw'}]} \omega_\Gmw = 0$ for $\Gmw' \neq \Gmw$ in $\mathcal L$.  In particular, the classes $\{ [\omega_\Gmw]/(2 \pi \sqrt{-1})^n \}_{\Gmw \in \mathcal L}$ form a basis for the top weight cohomology $\Gr_{2n}^W H^n(\cM_{1,n}; \Q)$ and the classes $\{[\mathfrak t_\Gmw]\}_{\Gmw \in \mathcal L}$ are a dual basis.
\end{proof}

\bigskip

\bigskip

\appendix
\section{Logarithmic forms, residues, and weight filtrations for DM stacks}

In this appendix, we briefly review logarithmic forms and weight filtrations in the context of simple normal crossings partial compactifications of smooth varieties, following standard constructions from \cite{Deligne71, KulikovKurchanov98, Voisin02b}.  Using the language of generalized $\Delta$-complexes, the constructions naturally generalize to the case of (not necessarily simple) normal crossings divisors from which it is easy to generalize to DM stacks.  We include, in particular, proofs of Propositions~\ref{prop:topweight} and \ref{prop:cyclesandboundaries}.  For additional background, details, and references regarding the extensions of de Rham cohomology and mixed Hodge structures to stacks, see \cite{Behrend04} and \cite{Teleman98, Dhillon06}, respectively.

\subsection{Logarithmic forms and weight filtrations} \label{app:logforms}

Let $V$ be a smooth variety of dimension $d$, let $D \subset V$ be a divisor with simple normal crossings, and let $U = V \smallsetminus D$, with $\iota\col U \hookrightarrow V$ the open inclusion.  We shall later require $V$ to be compact but for now it need not be.  We shall write $\mathcal{A}_U^q$ for the sheaf of complex-valued smooth (i.e.,\ $C^\infty$) differential $q$-forms on $U$ and recall that $\mathcal{A}_U^\bullet$ with the exterior differential forms a resolution of the constant sheaf $\C$ by acyclic sheaves.  In particular the cohomology of the induced complex of global sections
\[ 0 \rightarrow \Gamma(U, \mathcal A^0) \rightarrow \cdots \rightarrow \Gamma(U, \mathcal A^{2d}) \rightarrow 0
\] is canonically identified with the (sheaf cohomology and hence the) singular cohomology $H^*(U, \C)$.

There are subsheaves $W_m \mathcal{A}^q_V(\log D) \subset \iota_*\mathcal{A}_U^q$ on $V$ defined by the following local condition.  If $z_1, \dots, z_d$ are local holomorphic coordinates on an open subset of $V$, in which $D$ is defined by $z_1 \cdots z_k = 0$, then we require the $q$-form in these local coordinates to be of the form
\begin{equation*}
  \sum_I \omega_I \wedge \frac{dz_{i_1}}{z_{i_1}} \wedge \dots \wedge \frac{dz_{i_m}}{z_{i_m}},
\end{equation*}
where the sum is over all multi-indices $ I = (1 \leq i_1 < \dots < i_m \leq d)$ and the $\omega_I$ are sections of $\mathcal{A}^{q-m}_V$.  Adding over all $q$ gives subsheaves $W_m \mathcal{A}_V^\bullet(\log D) \subset \iota_* \mathcal{A}_U^\bullet$ which are closed under exterior derivative and wedge product with sections of $\mathcal{A}_V^\bullet$.  We have $W_m \mathcal{A}^q_V(\log D) \subset W_{m+1}\mathcal{A}^q_V(\log D)$ and we write $\mathcal{A}_V^q(\log D) = W_q \mathcal{A}_V^q(\log D)$ for the union.

It is easy to see that the inclusion $\mathcal{A}_V^\bullet(\log D) \hookrightarrow \iota_*\mathcal{A}_U$ induces a quasi-isomorphism on stalks, and since each $\mathcal{A}_V^q(\log D)$ is acyclic and $\mathcal{A}_U$ is $\iota_*$-acyclic (see, e.g. \cite[\S4.3]{KulikovKurchanov98}), it induces a quasi-isomorphism on global sections.  Hence the cohomology of the complex
\[ 0 \rightarrow \Gamma(V, \mathcal A^0_V(\log D)) \rightarrow \cdots \rightarrow \Gamma(V, \mathcal A^{2d}_V(\log D)) \rightarrow 0
\] is canonically identified with $H^*(U; \C)$.

The filtration of $\mathcal{A}^\bullet_V(\log D)$ by the subsheaves $W_m \mathcal{A}^\bullet(\log D)$ then gives rise to a filtration on $H^q(U;\C)$ in which $W_{m+q}H^q(U;\C)$ is the image of the $q$th cohomology group of $\Gamma(V,W_m \mathcal{A}^\bullet_V(\log D))$, i.e.,\ it consists of the cohomology classes which admit de Rham representatives which are global sections of $W_m\mathcal{A}^\bullet(\log D)$.  We shall write $\Gr^W H^\bullet(U;\C)$ for the associated graded object, i.e.,\ $\Gr^W_{m+q} H^q(U;\C) = W_{m+q} H^q(U;\C))/(W_{m+q-1}H^q(U;\C))$.  We may also pass to associated graded before taking cohomology and we shall write $\Gr^W_m \Gamma(V,\mathcal{A}_V^\bullet(\log D)) = \Gamma(V,W_m \mathcal{A}_V^\bullet(\log D)) / \Gamma(V,W_{m-1} \mathcal{A}_V^\bullet(\log D))$.  As usual the filtered chain complex gives rise to a spectral sequence, which in this case has
\begin{align*}
  E_\infty^{-p,q} & = \Gr^W_q H^{q-p}(U;\C)\\
  E_0^{-p,q} &= \Gr^W_p \Gamma(V,\mathcal{A}_V^{q-p}(\log D)),
\end{align*}
concentrated in the region $0 \leq p \leq d = \mathrm{dim}_\C(V)$ and $p \leq q \leq 2d + p$.  The differentials have bidegrees $d_r: E_r^{-p,q} \to E_r^{-p+r,q-r+1}$ and $d_0$ is induced by exterior derivation $d: \mathcal{A}_V^{q-p}(\log D) \to \mathcal{A}_V^{q+1-p}(\log D)$.  This spectral sequence is called the \emph{weight spectral sequence} and it converges because the filtration has finitely many steps.

\subsection{Residues and the weight spectral sequence}  \label{app:residues}

The $E_1$ page of the weight spectral sequence may be described explicitly using the \emph{Poincar\'e residue} homomorphism, which gives an isomorphism from $E_1^{-p,q}$ to the singular cohomology of the normalization of the closure of the codimension $p$ stratum $D^{(p)} \subset D$.  In the literature this is commonly defined in the case where $D \subset V$ has simple normal crossings and moreover there is chosen an ordering $D_1, \dots, D_r$ of its irreducible components.  Using the language of generalized $\Delta$-complexes, the definition is easily generalized to arbitrary normal crossing divisors provided we keep track of the monodromy.  We briefly give the definitions.

Recall that as part of our definition of $\Delta(D)$ as a generalized $\Delta$-complex in Section~\ref{sec:dual-compl-norm} we defined for each $p \geq 0$ a smooth variety $\FFF{p} \to D \subset V$, where $\FFF{0} \to D$ is the normalization and the closed points of $\FFF{p}$ are the $(p+1)$-tuples of distinct closed points in $\FFF{0}$ which lie over the same point in $D$.  The symmetric group $S_p$ acts freely on $\FFF{p-1}$ by permuting coordinates and the quotient $\FFF{p-1}/S_p$ is a smooth variety immersing into $V$.  In fact $\FFF{p-1}/S_p$ may be interpreted as the normalization of the closure of the codimension $p$ stratum $D^{(p)} \subset V$.  We shall write $\Z^\mathrm{sign}$ for the local system on $\FFF{p-1}/S_p$ corresponding to the sign representation $S_p \to \{\pm 1\}$ and similarly $\C^\mathrm{sign}$.  The sheaf $\mathcal{A}^{\bullet}_{\FFF{p-1}/S_p} \otimes_\C \C^\mathrm{sign}$ is then the sheaf of smooth differential forms on $\FFF{p-1}/S_p$ with values in the flat vector bundle $\C^\mathrm{sign}$.

The \emph{Poincar\'e residue} is a homomorphism
\begin{equation}\label{eq:7}
  \begin{aligned}
    E_0^{-p,q} \xrightarrow{\,\mathrm{res}\,} {}& {} (\Gamma(\FFF{p-1},
    \mathcal{A}^{q-2p}_{\FFF{p-1}}) \otimes_\C \C^\mathrm{sign})^{S_p}\\
    ={} & {} \Gamma(\FFF{p-1}/S_p, \mathcal{A}^{q-2p}_{\FFF{p-1}/S_p}
    \otimes_\C \C^\mathrm{sign}),
  \end{aligned}
\end{equation}
under which the differential $d_0$ in the spectral sequence corresponds to the exterior derivative in the de Rham complex of $\FFF{p-1}$, defined as follows.  If $z_1, \dots, z_d$ are local holomorphic coordinates on an open subset of $V$ in which $D$ is defined by $z_1 \cdots z_k = 0$ then each equation $z_i = 0$ determines a smooth branch of $D$ which lifts canonically to an open subset of $F([0])$.  Given distinct indices $i_0, \dots, i_{p-1} \in \{1, \dots, k\}$ we then obtain $p$ open subsets of $F([0])$ whose fiber product over $V$ defines an open subset $D_{(z_{i_0}, \dots, z_{i_{p-1}})} \subset \FFF{p-1}$ injecting to the subset of $D$ defined by the equation $z_{i_0} \cdots z_{i_{p-1}} = 0$.  By varying over local coordinates $(z_1, \dots, z_d)$ in which $D$ is defined by $z_1 \cdots z_k = 0$, the resulting open subsets $D_{(z_{i_0}, \dots, z_{i_{p-1}})}$ cover $\FFF{p-1}$.  Any weight $p$ differential $(q-p)$-form $\omega \in W_p\Gamma(V;\mathcal{A}_V^{q-p}(\log D))$ may locally be written as
\begin{equation*}
  \eta_{(i_0, \dots, i_{p-1})} \wedge \frac{dz_{i_0}}{z_{i_0}} \wedge \dots \wedge \frac{dz_{i_{p-1}}}{z_{i_{p-1}}} + \eta' + \sum_J \eta_J \frac{dz_J}{z_J},
\end{equation*}
where $\eta'$ has weight $p-1$ and the sum is over multiindices $J = (1 \leq j_0 < \dots < j_{p-1} \leq k)$ that are \emph{not} permutations of $(i_0, \dots, i_{p-1})$.  The residue $\mathrm{res}(\omega) \in \Gamma(\FFF{p-1}, \mathcal{A}^{q-2p}_{\FFF{p-1}})$ may be defined by the local formula
\begin{equation*}
  \mathrm{res}(\omega) \vert_{D_{(z_{i_0}, \dots, z_{i_{p-1}})}} = \eta_{(i_0, \dots, i_{p-1})} \vert_{D_{(z_{i_0}, \dots, z_{i_{p-1}})}},
\end{equation*}
which is easily seen to be well defined and to give an $S_p$ invariant element of $\Gamma(\FFF{p-1}, \mathcal{A}^{q-2p}_{\FFF{p-1}}) \otimes_\C \C^\mathrm{sign}$.

Just as in the case of simple normal crossings, the residue homomorphism~(\ref{eq:7}) sends the $d_0$ differential in the spectral sequence to the exterior derivative of forms, and is in fact a quasi-isomorphism.  Hence it induces an isomorphism
\begin{equation}\label{eq:6}
  \begin{aligned}
    E_1^{-p,q} \xrightarrow{\,\mathrm{res}\,} {} & {} (H^{q-2p}(\FFF{p-1};\C)
    \otimes \Z^\mathrm{sign})^{S_p}\\  ={} & {} H^{q-2p}(\FFF{p-1}/S_{p};
    \C^\mathrm{sign}).
  \end{aligned}
\end{equation}
Since $\dim_\R(\FFF{p-1}/S_p) = \dim_\R(V) - 2p = 2d - 2p$ we see that $E_1^{-p,q} = 0$ for $q > 2d$ so there are no differentials coming into the row $E_r^{-*,2d}$ for $r \geq 2$.  Hence we obtain an edge homomorphism
\begin{equation*}
  H^{2d-p}(U;\C) \twoheadrightarrow \Gr_{2d}^W H^{2d-p}(U;\C) = E_\infty^{-p,2d} \hookrightarrow E_2^{-p,2d}.
\end{equation*}

Assuming now in addition that $V$ is compact we can say more about the row $E_2^{*,2d}$ and in fact the whole spectral sequence.  Indeed, if $V$ is compact then the smooth variety $\FFF{p-1}$ is also compact, and hence we have Poincar\'e duality isomorphisms
\begin{equation*}
  E_1^{-p,q} = H^{q-2p}(\FFF{p-1}/S_{p}; \C^\mathrm{sign}) \cong H_{2d-q}(\FFF{p-1}/S_p;\C^\mathrm{sign}).
\end{equation*}
In particular for $q=2d$ we have 
\begin{equation*}
  E_1^{-p,2d} = H^{2d-2p}(\FFF{p-1}/S_{p}; \C^\mathrm{sign}) \cong H_0(\FFF{p-1}/S_p;\C^\mathrm{sign}) = \widetilde C_{p-1}(\Delta(D);\C)
\end{equation*}
where $C_{p-1}(\Delta(D);\C)$ is the cellular chains of the generalized $\Delta$-complex $\Delta(D)$, as defined in Section~\ref{sec:cellular-chains}.  A diagram chase shows that the isomorphism sends the differential $d_1$ in the spectral sequence to the boundary map in the cellular chains, and hence induces an isomorphism
\begin{equation*}
  E_2^{-p,2d} \cong \widetilde H_{p-1}(\Delta(D);\C).
\end{equation*}
Finally, the Hodge filtration of $\mathcal{A}_V^\bullet(\log D)$ gives a mixed Hodge structure on the entire spectral sequence, and when $V$ is compact the rows of the $E_1$ page are pure of different weight.  Hence the same is true for the $E_2$ page, and since no differential can exist between pure Hodge structures of different weight the spectral sequence must collapse on the $E_2$ page.  Hence in this case the edge homomorphism
\begin{equation*}
  \Gr_{2d}^W H^{2d-p}(U;\C) = E_\infty^{-p,2d} \hookrightarrow E_2^{-p,2d} = \widetilde H_{p-1}(\Delta(D);\C)
\end{equation*}
is an isomorphism, proving Proposition~\ref{prop:topweight}.

The edge homomorphism on the $E_1$ page of the spectral sequence described above is induced by an explicit composition
\begin{equation*}
  W_p \Gamma(V, \mathcal{A}_V^{2d-p}(\log D)) \twoheadrightarrow E_0^{-p,2d} \xrightarrow{\,\mathrm{res}\,} \Gamma(\FFF{p-1};\mathcal{A}^{2d-2p}) \xrightarrow{\,\int\,} \widetilde C_{p-1}(\Delta(D);\C),
\end{equation*}
where $\int$ is the following homomorphism.  Recall that $C_{p-1}(\Delta(D);\C)$ is spanned over $\C$ by symbols $[\cY]$, where $\cY \subset \FFF{p-1}$ runs through path components, subject to the relation $[\sigma \cY] = \mathrm{sgn}(\sigma) [\cY]$ for $\sigma \in S_p$.  Then $\int$ sends a top degree differential form $\eta$ on $\FFF{p-1}$ the chain $\frac1{p!} \sum_\cY (\int_\cY \eta)[\cY] \in C_{p-1}(\Delta(D);\C)$, where the sum is over path components $\cY \subset \FFF{p-1}$.  For a weight $p$ form $\omega \in \Gamma(V, \mathcal{A}_V^{2d-p}(\log D))$ we can view $\mathrm{res}(\omega)$ as a top differential form on the stratum $D^{(p)} \subset D$ with values in the flat vector bundle $\C^\mathrm{sign}$, extending to a form on the normalization of the closure of this stratum.  Hence we shall also write $\omega \mapsto \int_{D^{(p)}} \res(\omega) \in C_{p-1}(\Delta(D);\C)$ for this composition.

\subsection{De Rham cohomology of smooth DM stacks}

As remarked earlier, the definition of $\Delta(D)$ as a generalized $\Delta$ complex applies equally well in the case of a normal crossings divisor $\mathcal{D}$  in a proper DM stack $\ocX$.  The weight spectral sequence and the residue integral isomorphism from $\Gr^W_p H^{q-p}(\mathcal{X};\C)$ to $\widetilde H_{p-1}(\Delta(\mathcal{D});\C)$ are easily generalized to this setting.

Recall that the de Rham complex of $\ocX$ is defined using an etale atlas $V \to \ocX$ as the equalizer
\begin{equation*}
  \Gamma(\ocX, \mathcal{A}^\bullet_{\ocX}) \to \Gamma(V, \mathcal{A}^\bullet_V) \double \Gamma(V \times_{\ocX} V, \mathcal{A}^\bullet_{V \times_{\ocX} V})
\end{equation*}
and similarly for $\mathcal{A}^\bullet_{\mathcal{X}}$ and the subsheaves $W_m \mathcal{A}^\bullet_{\ocX}(\log \mathcal{D}) \subset \mathcal{A}^\bullet_{\ocX})$.

Just as in the special case of a variety $\overline X$ with a simple normal crossing divisor $D$, discussed above, the sheaves $\mathcal A^\bullet_{\ocX}(\log \mathcal D)$ are acyclic, the inclusion $\mathcal A^\bullet_{\ocX}(\log D) \subset \iota_* \mathcal A^\bullet(\cX)$ is a quasi-isomorphism, and $\mathcal A^\bullet (\cX)$ is $\iota_*$-acyclic, where $\iota\col \cX \hookrightarrow \ocX$ is the open immersion.  Therefore, the cohomology of the complex
\[ 0 \rightarrow \Gamma (\ocX, \mathcal A^0_{\ocX} (\log \cD)) \rightarrow \cdots \rightarrow \Gamma (\ocX, \mathcal A^{2d}_{\ocX}(\log \cD)) \rightarrow 0
\] is naturally identified with $H^*(\cX, \C)$.

With these definitions, the construction and properties of the weight spectral sequence carries through verbatim to the case of DM stacks.  To phrase the result in terms of residue integrals, let us also recall that the integral of a compactly supported top degree differential form $\omega$ on a global quotient stack $V = [U/G]$, is defined as $\int_V \omega = \frac1{|G|} \int_U p^* \omega$, where $p: U \to V$ is the quotient map.  On stacks that are not global quotients the integral is defined using partitions of unity.  If $V$ is a smooth proper DM stack of complex dimension $d$, possibly with several components, the integral then gives an isomorphism $H^{2d}_{\mathrm{dR}}(V;\C) \to H_0(V;\C)$, just as in the case of varieties.

As for varieties, the spectral sequence again collapses at the $E_2$ page and the Poincar\'e residue gives an isomorphism $E_1^{-p,2d} \to \widetilde C_{p-1}(\Delta(\cD);\C)$, assuming of course that $\ocX$ is proper.  We use this to prove Proposition \ref{prop:cyclesandboundaries}, as follows. The isomorphism
\begin{align*}
  E_0^{-p,2d} & = \Gr^W_p \Gamma(\ocX, \cA^{2d-p}_{\ocX}(\log \cD))\\
  &= W_p \Gamma(\ocX, \cA^{2d-p}_{\ocX}(\log \cD))
    / W_{p-1} \Gamma(\ocX, \cA^{2d-p}_{\ocX}(\log \cD))
\end{align*}
and the fact that $E_1$ is the cohomology of $E_0$, we see that any chain
$F \in \widetilde C_{p-1}(\Delta(\cD);\C)$ is of the form
$\int_{\cD^{(p)}} \mathrm{res}(\omega)$ for some $\omega$ which represents a class in $E_1^{-p,2d}$, i.e.,\ $\omega$ has weight $p$ but $d\omega$ has weight $p-1$.  This proves (i).  The fact that the residue integral induces an isomorphism $E_1^{-p,2d} \cong \widetilde C_{p-1}(\Delta(\cD);\C)$ easily implies (iii).  Claim (ii) is slightly more subtle; it may be deduced from the collapse of the spectral sequence, as follows.  Let $F \in \widetilde C_{p-1}(\Delta(\cD);\C)$ be a cycle and write $F = \int_{\cD^{(p)}} \mathrm{res}(\omega)$ for some $\omega$ with  weight $p$ and $d\omega$ weight $p-1$.  Then the class $[\omega] \in E_1$ survives to $E_2$ because $F$ is a cycle, by the chain isomorphism $E_1^{-p,2d} \cong \widetilde C_{p-1}(\Delta(\cD);\C)$.  By construction of the spectral sequence associated to a filtered cochain complex, the class $[\omega]$ survives to $E_2$ if and only if there is a representative differential form $\omega$ such that $d\omega$ has weight $p-2$.  Similarly, if we assume inductively that $[\omega]$ survives to represent a class on the $E_r$ page, then the differential $d^r[\omega]$ is the obstruction to choosing a representative differential form with $d\omega$ having weight $p-(r+1)$.  Therefore the collapse of the spectral sequence implies that any $[\omega]$ which survives to $E_2$ survives to $E_\infty$, and hence that if $F \in \widetilde C_{p-1}(\Delta(\cD);\C)$ is a cycle, there is no obstruction to finding a \emph{closed} $\omega$ of weight $p$ with $F = \int_{\cD^{(p)}} \mathrm{res}(\omega)$, proving (ii).

\subsection{Top weight cohomology with rational coefficients} \label{app:topweight} We now give a second proof of Proposition~\ref{prop:topweight}, identifying the reduced rational homology of the boundary complex with the top weight cohomology of the stack, which does not involve logarithmic forms and residues.  Instead, we closely follow the proof for varieties given in \cite[Theorem~4.4]{boundarycx}.  The one additional fact needed is that the cohomology of smooth DM stack $\cY$ with projective coarse moduli space $Y$ is pure.  To see this, note that the natural map $\cY \to Y$ induces an isomorphism $H^*(Y;\Q) \to H^*(\mathcal{Y};\Q)$ (see \cite{Behrend04} or \cite[Theorem~4.40]{Edidin13}) and, since $Y$ is a K\"ahler $V$-manifold, its cohomology is pure \cite[Theorem~2.43]{PetersSteenbrink08}.

\begin{proof}[Proof of Proposition~\ref{prop:topweight}]
After a finite sequence of blow-ups, we may assume that the irreducible components $\cD_1,\ldots,\cD_r$ of $\cD$ are smooth.  Then there is a simplicial stack whose $(k-1)$-simplices are given by the disjoint union
\begin{equation}\label{eq:1}
\coprod_{1 \leq i_0 \leq \dots \leq i_{k-1} \leq r} \cD_{i_0} \times_{\ocX} \dots \times_{\ocX} \cD_{i_{k-1}}.
\end{equation}
The resulting spectral sequence converges to $H^*(\cD;\Q)$ and has $E_1^{p,q}$ given by $H^q(-;\Q)$ of~\eqref{eq:1}.  See \cite[5.2.1.1]{Deligne74b} for more details on this spectral sequence. The entire spectral sequence preserves the weight filtration, and we obtain a spectral sequence converging to $W_0H^*(\cD;\Q)$ whose $E_1^{p,q}$ is given by $W_0H^q(-;\Q)$ of~\eqref{eq:1}.  Then, since each $\cD_{i_0} \times_{\ocX} \dots \times_{\ocX} \cD_{i_p}$ in~\eqref{eq:1} is a smooth and proper DM stack with projective coarse moduli space, its mixed Hodge structure is pure, as noted above.  In particular, the weight zero cohomology of $\cD$ is concentrated in degree zero.  The spectral sequence is therefore concentrated on the line $E^1_{*,0}$, where it is isomorphic to the cellular chain complex for $\Delta(\cD)$.  This proves
\begin{equation*}
      W_0 H^k(\cD; \Q) \cong H^k(\Delta(\cD); \Q).
\end{equation*} The long exact pair sequence for $(\ov \cX, \cD)$ and Poincar\'e duality identify $W_0 H^k(\cD; \Q)$ with $\Gr_{2n}^W H^{2n-k-1}(\cX; \Q)$, and the theorem follows.
\end{proof}

\subsection{Rational homology of $\|\Delta(D_\bullet)\|$}\label{sec:simplicialproof}  

Finally, we prove Proposition~\ref{p:simplicial}, that for a normal crossings compactification $D\subset X$, the natural map $\|\Delta(D_\bullet)\| \rightarrow |\Delta(D)|$ induces an isomorphism in rational homology, where as $\|\Delta(D_\bullet)\|$ is as discussed in \S\ref{sec:simplicialinvariance}.

\begin{proof}
We shall study the simplicial abelian group $[q] \mapsto C_p(\Delta(D_q);\Q)$ for each fixed $p$.  There is an associated chain complex
\begin{equation*}
    \dots \to C_p(\Delta(D_q);\Q) \xrightarrow{\partial}
    C_p(\Delta(D_{q-1});\Q) \to
    \dots \to C_p(\Delta(D_0);\Q) \to C_p(\Delta(D);\Q) \to 0,
\end{equation*}
in which $\partial = \sum_{i=0}^q (-1)^i (d_i)_*$.  It suffices to prove that this chain complex is acyclic for each $p$ and in fact we shall construct an explicit chain contraction.

We shall write $V_{-1} = X$ and $d_0: V=V_0 \to V_{-1} = X$ and similarly for $D_{-1} = D$.  We will use the fact that $d_0: V_{q} \to V_{q-1}$ is an \'etale map and that $D_{q} = V_{q} \times_{V_{q-1}} D_{q-1}$ to produce ``transfer'' maps
\begin{equation*}
    \tau: C_p(\Delta(D_{q-1});\Q) \to C_p(\Delta(D_{q});\Q)
\end{equation*}
for all $q \geq 0$ with the properties that $d_0 \circ \tau = \mathrm{id}$ and $d_i \circ \tau = \tau \circ d_{i-1}$ for $i = 1, \dots, q+1$.  This implies $\partial \circ \tau + \tau \circ \partial = \mathrm{id}$ giving the required chain contraction.

To construct this operation $\tau$, let $E \to D$ be the normalization and let $$\FFF{p} \subset E \times_X \dots \times_X E$$ be as in \S\ref{sec:dual-compl-norm}, with the symmetric groups $S_{p+1}$ acting by permuting the coordinates. Recall that the $\Q$ vector space $C_p(\Delta(D);\Q)$ is generated by symbols $[C]$ where $C \subset \FFF{p}$ is an irreducible component, subject to the relations $[\sigma C] = \mathrm{sgn}(\sigma)[C]$ for $\sigma \in S_{p+1}$. Similarly we have normalizations $E_q \to D_q$ and subvarieties $\FFFF{q}{p} \subset E_q \times_{V_q} \dots \times_{V_q} E_q$ of the $(p+1)$-fold fiber product and $C_p(\Delta(D_q);\Q)$ is generated by symbols $[C]$ where $C \subset \FFFF{q}{p}$ is an irreducible component.  By the same argument as in the proof of Lemma~\ref{lemma:etale-descent}, we have a pullback diagram
\begin{equation*}
    \xymatrix{
      \FFFF{q}{p} \ar[r] \ar[d]_{d_0} & V_q \ar[d]^{d_0}\\
      \FFFF{q-1}{p} \ar[r] & V_{q-1},
    }
\end{equation*}
and in particular $\FFFF{q}{p} \to \FFFF{q-1}{p}$ is an \'etale map between smooth (but likely disconnected) varieties.  For each component $C \subset \FFFF{q-1}{p}$ the inverse image in $\FFFF{q}{p}$ is a disjoint union of smooth components $C'_1, \dots, C'_r$ and to each of them we attach the number $m_j \in \Z_{> 0}$ defined as the degree of the field extension $\kappa(C) \subset \kappa(C'_j)$ and let $w_j = m_j/(\sum_k m_k)$.  We then define $\tau: C_p(\Delta(D_{q-1});\Q) \to C_p(\Delta(D_q);\Q)$ by sending
$[C]$ to $\sum w_j [C'_j]$.

It is then clear that $d_0 \tau([C]) = d_0 (\sum w_i [C'_i]) = \sum w_i [C] = [C]$ and it remains to see that $d_i \tau([C]) = \tau (d_{i-1}[C])$ for $i > 0$, which may be deduced from the pullback diagram
\begin{equation*}
      \xymatrix{
        \FFFF{q}{p} \ar[r]^{d_i} \ar[d]_{d_0} & \FFFF{q-1}{p} \ar[d]^{d_0}\\
        \FFFF{q-1}{p} \ar[r]_{d_{i-1}} & \FFFF{q-2}{p}.  }
\end{equation*}
Indeed, if the inverse image of $C'' = d_{i-1}(C) \subset \FFFF{q-2}{p}$ is the disjoint union of components $C'_1, \dots, C'_r \subset \FFFF{q-1}{p}$ then the inverse image of $C$ itself is the disjoint union of $C \times_{C''} C'_1, \dots, C \times_{C''} C'_1 \subset \FFFF{q}{p}$. Then $\tau([C])$ is a weighted average of the components of $\coprod_i C \times_{C''} C'_i$ and $d_i \tau([C])$ is a weighted average of their images in $\FFFF{q-1}{p}$ which is a weighted average the $C'_i$.  The class $\tau(d_{i-1})([C]) = \tau(C'')$ is also a weighted average of the $C'_i$ and it remains to see that the weights agree.  This follows from the fact that the ring $\kappa(C) \otimes_{\kappa(C'')} \kappa(C'_i)$ splits as a product of field extensions of $\kappa(C)$ whose degrees add up to the degree of the field extension $\kappa(C'') \to \kappa(C'_i)$.
\end{proof}

\bibliographystyle{amsalpha}
\bibliography{math}

\end{document}